\documentclass{siamart190516}

\let\counterwithout\relax

\usepackage{lipsum}
\usepackage{amsfonts,amssymb,amsmath}
\usepackage{graphicx}
\usepackage{epstopdf}
\usepackage{algorithmic}
\usepackage{tikz}
\usetikzlibrary{shapes,arrows}
\usepackage{amsopn}
\usepackage{cite}

\usepackage{chngcntr}
\counterwithout{equation}{section} 

\title{Stochastic Conditional Gradient++: (Non-)Convex Minimization and Continuous Submodular Maximization\thanks{The authors are listed in alphabetical order.}}
\usepackage{bbm}
\usepackage{xspace}

\usepackage{pifont}

\usepackage{hyperref}

\usepackage{amsmath}
\usepackage{amssymb}

\usepackage{epstopdf}

\usepackage{graphicx}

\usepackage{algorithm}
\usepackage{algorithmic}

\usepackage{xcolor}

\usepackage{comment}

\input{required.tex}

\newcommand{\OPT}{\text{OPT}}
\newcommand{\SCGPP}{{$\text{SCG}{++} $}\xspace}
\newcommand{\NMSCGPP}{{$\text{SMCG}{++} $}\xspace}
\newcommand{\FWPP}{{$\text{SFW}{++} $}\xspace}
\newcommand{\scg}{{\texttt{Stochastic Continuous Greedy}{++}}\xspace}
\newcommand{\nmscg}{{\texttt{Stochastic Measured Continuous Greedy}{++}}\xspace}
\newcommand{\sfw}{{\texttt{Stochastic Frank-Wolfe}{++}}\xspace}

\usepackage{lipsum}
\usepackage{amsfonts}
\usepackage{graphicx}
\usepackage{epstopdf}
\usepackage{algorithmic}
\ifpdf
  \DeclareGraphicsExtensions{.eps,.pdf,.png,.jpg}
\else
  \DeclareGraphicsExtensions{.eps}
\fi

\newsiamremark{remark}{Remark}
\newsiamremark{hypothesis}{Hypothesis}
\crefname{hypothesis}{Hypothesis}{Hypotheses}
\newsiamthm{claim}{Claim}

\headers{Stochastic Conditional Gradient++}{H. Hassani, A. Karbasi, A. Mokhtari, and Z. Shen}

\title{Stochastic Conditional Gradient++: (Non-)Convex Minimization and Continuous Submodular Maximization\footnote{T\lowercase{he authors are listed in alphabetical order}.}
	\thanks{\lightblue{A subset of this work (Section 4.1, Theorem 5.3 of Section 5.1, and Section 5.2) appeared in \cite{NIPS2019_9466}, which contains the results on maximizing monotone DR-submodular functions.}}
}
\author{Hamed Hassani\thanks{University of Pennsylvania 
  (\email{hassani@seas.upenn.edu}).}
\and Amin Karbasi\thanks{Yale University
  (\email{amin.karbasi@yale.edu}).}
\and Aryan Mokhtari\thanks{The University of Texas at Austin
	(\email{mokhtari@austin.utexas.edu}).}
\and Zebang Shen\thanks{University of Pennsylvania
	(\email{zebang@seas.upenn.edu}).}}

\begin{document}
\maketitle
\begin{abstract}
In this paper, we consider the general \textit{non-oblivious} stochastic optimization where the underlying stochasticity may change during the optimization procedure and  depends on the point at which the function is evaluated. We develop \sfw (\FWPP), an efficient variant of the conditional gradient method for
minimizing a  smooth non-convex  function  subject to a convex body constraint. We show that \FWPP converges to an $\epsilon$-first order stationary point by using $O(1/\epsilon^3)$ stochastic gradients. Once further structures are present, \FWPP's theoretical guarantees, in terms of the convergence rate and  quality of its solution,   improve. In particular, for minimizing a convex  function, \FWPP achieves an $\epsilon$-approximate optimum while using   $O(1/\epsilon^2)$ stochastic gradients. It is known that this rate is optimal  in terms of stochastic gradient evaluations.  Similarly, for maximizing a monotone continuous DR-submodular function, a slightly  different form of \FWPP, called \scg (\SCGPP), achieves a tight  $[(1-1/e)\OPT -\epsilon]$ solution while using $O(1/\epsilon^2)$ stochastic gradients. Through an information theoretic argument, we also prove that \SCGPP's convergence rate is optimal. Finally, for maximizing a non-monotone continuous DR-submodular function, we can achieve a  $[(1/e)\OPT -\epsilon]$ solution by using $O(1/\epsilon^2)$ stochastic gradients. We should highlight that our results and our novel variance reduction technique trivially extend to the standard and easier \textit{oblivious} stochastic optimization settings for (non-)convex and continuous submodular settings.

%
%
\end{abstract}

%

\vspace{-1mm}

\vspace{-1mm}
\vspace{-1mm}

\vspace{-1mm}
\section{Introduction}
In this paper, we consider the following \emph{non-oblivious} stochastic maximization problem:
\vspace{-1mm}
\begin{equation}
	\max_{\xB \in \CM} F(\xB) \ :=\ \max_{\xB \in \CM}\ \EBB_{\zB\sim p(\zB; \xB)} [\tilde{F}(\xB; \zB)],\vspace{-1mm}
	\label{eqn_DR_submodular_maximization_loss}
\end{equation}
where $\xB\in \RBB_+^d$ is the decision variable, $\CM\subseteq \RBB^d$ is a  feasible set, $\zB\in\ZM$ is a random variable with distribution $p(\zB;\xB)$, and the function $F:\RBB^d \rightarrow \RBB$ is defined as the expectation over a set of \blue{smooth} stochastic functions $\tilde{F}:\RBB^d\times\ZM \rightarrow \RBB$. 
Problem \eqref{eqn_DR_submodular_maximization_loss} is called non-oblivious as the underlying distribution depends on the variable $\xB$ and may change during the optimization procedure. Note that the standard stochastic (convex/non-convex) optimization is a special case of Problem~\eqref{eqn_DR_submodular_maximization_loss}. We focus on providing efficient solvers for Problem \eqref{eqn_DR_submodular_maximization_loss} in terms of the sample complexity of $\zB$ (a.k.a., calls to the stochastic oracle),  where $F$ is (non-)concave or continuous submodular and the feasible set $\CM$ is a bounded convex body. Note that maximizing a non-concave function $F$ is equivalent to minimizing a non-convex function $-F$. However, in order to unify the language between the non-convex minimization and  continuous submodular maximization, we resort to  formulation \eqref{eqn_DR_submodular_maximization_loss}. In the following, we discuss three concrete instances of non-oblivious stochastic optimization, namely, multi-linear extension of a discrete submodular function, Maximum a Posteriori (MAP) inference in determinantal point processes, and policy gradient in reinforcement learning. In all these problems the stochasticity of the objective function crucially depends on the decision variable $\xB$ at which we evaluate the function. 

\textbf{Our contributions.} In this paper, we develop a new variance reduction technique for non-oblivious stochastic optimization problem~\eqref{eqn_DR_submodular_maximization_loss}. 
\blue{Note that the success of the variance reduction technique in the standard oblivious stochastic setting relies on the property that the difference of gradients at two points can unbiasedly estimated using a single sample.
However, for the more general non-oblivious case, this crucial property is missing, which invalidates the applicability of the previous arts in the problems of interest here (we elaborate this at the end of Section \ref{section_preliminaries}).
The key algorithmic contribution of this paper is a way to estimate the difference of gradients without introducing extra bias, as discussed in detail in Section \ref{section_stochastic_gradient_approximation}.
This draws a clear distinction of our work from the literature on stochastic first-order algorithms.
}
In particular, we show the following results for problem~\eqref{eqn_DR_submodular_maximization_loss}.

\begin{itemize}
\item For maximizing a general non-concave function $F$ (or minimizing a non-convex function), we develop \sfw (\FWPP) that converges to an $\epsilon$-first order stationary point using $O(1/\epsilon^3)$ stochastic gradients in total. Our result improves upon the previously best convergence rate of $O(1/\epsilon^4)$ by \cite{reddi2016stochastic} in the oblivious stochastic setting. Moreover, as a by-product, \FWPP provides the first trajectory complexity of $O(1/\epsilon^3)$ for policy gradient methods in reinforcement learning with convex constraints.  

\item When the function $F$ is concave, \FWPP achieves an $\epsilon$-approximate optimum while using   $O(1/\epsilon^2)$ stochastic gradients, thus achieving the optimum rate for this instance. This result improves upon the  previously best convergence rate of $O(1/\epsilon^3)$ in \cite{mokhtari2018stochastic}.  In the oblivious stochastic setting, the convergence rate of \FWPP is on par with the stochastic gradient sliding \cite{lan2016conditional}.

\item For maximizing a monotone DR-continuous function $F$, the \scg (\SCGPP) method is introduced, the first algorithm that achieves the tight $[(1-1/e)\OPT -\epsilon]$ solution by using $\OM({1}/{\epsilon^2})$ stochastic gradients in total. Through an information theoretic argument, we also show that no first-order algorithm can achieve a converegnce rate better than $\OM({1}/{\epsilon^2})$. This result improves upon the  previously best convergence rate of $O(1/\epsilon^3)$ in \cite{mokhtari2018conditional}.  Moreover, \blue{\SCGPP} leads to the fastest method for maximizing a multi-linear extension of a monotone submodular set function.

\item For maximizing a non-monotone DR-continuous function $F$, subject to a down-closed convex body $\CM$, we develop \nmscg (\NMSCGPP) that achieves a $[(1/e)\OPT -\epsilon]$ solution  by using at most $\OM({1}/{\epsilon^2})$ stochastic gradients. This result improves upon the  previously best convergence rate of $O(1/\epsilon^3)$ in \cite{mokhtari2018stochastic}.  Moreover, \NMSCGPP (along with lossless rounding schemes such as contention resolution, randomized pipage rounding, etc)  provides a rigorous $1/e$ approximation guarantee for MAP estimation of a determinantal point process, improving upon the semi heuristic $1/4$ approximation guarantee in \cite{gillenwater2012near}.
\end{itemize}

\subsection{Examples} In this subsection, we briefly mention some instances of the non-oblivious optimization problem in~\eqref{eqn_DR_submodular_maximization_loss}.
\vspace{2mm}

\noindent{\textbf{Multi-Linear Extension of a Discrete Submodular Set Function.}}
One canonical example of the stochastic optimization problem in 
~\eqref{eqn_DR_submodular_maximization_loss} is the multi-linear extension of a discrete 
submodular function.
Specifically, consider a discrete submodular set function $f: 2^V \rightarrow \RBB_+$ defined over the set $V$. 
The aim is to solve $\max_{S \in \mathcal{I}} f(S)$ where $\mathcal{I}$ encodes a matroid constraint. For this case, the greedy algorithm leads to a $1/2$ approximation guarantee, but one can achieve the optimal approximation guarantee of $(1-1/e)$ by maximizing its multilinear extension $F: [0,1]^V \to \mathbb{R}_+$, defined as
\vspace{-1mm}
\begin{equation} \label{multilinear_general} 
	F(\xB)\ :=\ \mathbb{E}_{\zB \sim \xB}[f(\zB(\xB))] \ :=\ \sum_{S \subseteq V } f(S) \prod_{i \in S} x_i 
	\prod_{j \notin S} (1-x_j).\vspace{-1mm}
	\end{equation} 
Here, each element $e$ of the random set $\zB(\xB)$ is sampled with probability $x_{e}$. 
This 
problem 
is an instance  of \eqref{eqn_DR_submodular_maximization_loss} if we define $\tilde{F}(\xB,\zB)$ as 
$f(\zB(\xB))$, the joint probability $p(\xB, \zB)$ as the distribution of the random set $\zB(\xB)$ 
(i.e., 
each coordinate 
$z_e$ is generated according to a Bernoulli distribution with parameter $x_e$), and the set $\CM$ 
as 
the convex hull of $\mathcal{I}$. Later, we show how constraint submodular maximization 
can be solved efficiently, providing a fast method for maximizing a multi-linear extension 
function. 
\vspace{2mm}

\noindent\textbf{MAP Inference in Determinental Point Processes (DPPs).} DPPs are a class of 
discrete 
probabilistic models that were introduced in statistical physics and random matrix theory. 
Due to their ability to model repulsion and negative correlations, they have shown to be key 
concepts for many applications in  machine learning \cite{kulesza2012determinantal}. Formally, given a positive definite matrix $A$ of size  $n$, a DPP assigns to any subset $S \subseteq [n] \triangleq \{1,2,\cdots, 
n\}$, a 
probability value $\text{Pr}(S) = \text{det}(A_S)/\text{det}(A)$. MAP inference in such a model 
consists 
of maximizing the value $\text{det}(A_S)$, or equivalently the log-likelihood $\log 
\text{det}(A_S)$, over all the subsets $S \subseteq [n]$. Indeed, the set function $f(S) = \log 
\text{det}(A_S)$ is submodular but generally non-monotone.   As a result, MAP inference in DPPs is 
an 
instance 
of a non-monotone submodular maximization problem. Moreover, MAP inference may be restricted 
to 
subsets that satisfy some given constraints. For  instance, for cardinality constraint, the 
problem is to find a subset with a size of at most $k$, which has the largest probability. There are in 
general 
two approaches to maximize the log-likelihood function in DPPs subject to feasibility constraints, 
both of which 
rely 
on appropriate continuous extensions.  The first approach is to  form the multi-linear extension, 
defined in \eqref{multilinear_general}, and solve the resulting constrained non-monotone 
submodular optimization problem. Since the multi-linear extension involves summing over 
exponentially many terms, it was generally believed that the optimization will be computationally 
expensive and convergence issues may arise. This paper overcomes these challenges completely 
by providing  the first $((1/e)\OPT-\epsilon)$ solution in $O(1/\epsilon^2)$ stochastic 
iterations, whenever  the sampled sets form a matroid and the function $\log 
\text{det}(A_S)$ is non-negative for all $S$.   Another approach,  proposed by 
\cite{gillenwater2012near}, is to form the so-called 
\emph{softmax} extension $G: [0,1]^n \to \mathbb{R}$ defined as
\vspace{-2mm}
\begin{equation*} G(\xB) = \mathbb{E}_{\zB\sim \xB}\left[\exp(f(\zB(\xB)))\right]  = \log \text{det} 
\left(\text{diag}(\xB) 
(A - I) + I\right).
\vspace{-1mm}
\end{equation*}
The soft-max extension is a deterministic function that can be maximized within a $1/4$ 
approximation to the optimal value $\OPT$. Unlike the multi-linear extension, no provable 
rounding scheme is known for the soft-max extension. Therefore, our approach not only improves the approximation ratio of the MAP estimator, but also enjoys a rigorous end-to-end guarantee for this problem. 
%
\vspace{2mm}

\noindent{\textbf{Reinforcement Learning.}}
Consider a discrete time index $h\geq0$ and a Markov system with states $s_h \in 
\SM$ and actions $a_h\in\AMM$. The probability distribution of the initial state is $\rho(s_0)$ and the 
conditional probability distribution of transitioning into $s_{h+1}$ given that we are in state $s_h$ 
and take action $a_h$ is $\PM(s_{h+1}| s_h, a_h)$. Actions are chosen based on a random policy $\pi$ in which $\pi(a_h|s_h)$ is the distribution for taking action $a_h$ when 
observing state $s_h$. We assume that policies are parametrized by a vector $\theta \in \RBB^{d}$ 
and use $\pi_{\theta}$ as a shorthand for the conditional distribution $\pi(a_h|s_h; \theta)$ 
associated to $\theta$. For a given time horizon $H$ we define the trajectory $\tau :=(s_1, a_1, 
\ldots, s_{H}, a_{H})$ as the collection of state-action pairs experienced up until time $h=H$. 
Given the initial distribution $\rho(s_0)$, the transition kernel $\PM(s_{h+1}|s_h, a_h)$, and the 
Markov property of the system, it follows that the probability distribution over trajectories $\tau$ is
\vspace{-2.5mm}
\begin{equation} \label{eqn_distributioin_of_trajectory}
   p(\tau;\pi_\theta) 
      \defi \rho(s_0)\, \prod_{h=1}^{H}\, \PM(s_{h+1}|s_h,a_h)\,\pi(a_h|s_h).\vspace{-2.5mm}
\end{equation}
Associated with a state action pair we have a reward function $r(s_h, a_h)$. When following a trajectory $\tau=(s_1, a_1, \ldots, s_{H}, a_{H})$, we consider the accumulated reward discounted by a geometric factor $\gamma$ 
$	\RM(\tau) = \sum_{h=1}^{H} \gamma^h r(s_h, a_h)$.
Our goal in reinforcement learning is to find the policy parameter $\theta$ that maximizes the expected reward
\vspace{-4mm}
\begin{equation}	\label{eqn_cost_to_optimize}
	\max J(\theta) \defi \EBB_{\tau\sim p(\tau; \pi_\theta)}[ \RM(\tau)] = \int \RM(\tau) p(\tau;\pi_\theta) \dB \tau.\vspace{-2mm}
\end{equation}
Here, the underlying distribution $p$ depends on the variable $\theta$ and therefore this problem 
can be considered as an instance of the non-oblivious formulation 
in~\eqref{eqn_DR_submodular_maximization_loss}. 

To find an $\epsilon$-first order stationary point of problem \eqref{eqn_cost_to_optimize}, the trajectory complexities of classic SGD-based methods like REINFORCE are $\OM(1/\epsilon^4)$ \cite{sutton2018reinforcement}.
While a direct application of the recent variance reduced gradient estimation leads to a biased gradient 
estimator of $\JM(\cdot)$, due to the inherent difficulty of the non-oblivious optimization of 
\eqref{eqn_cost_to_optimize},
a recent work by \cite{shen2019hessian} proposed a policy Hessian method to aid the estimation of 
the policy gradient.
They improve the trajectory complexity from $\OM(1/\epsilon^4)$ to $\OM(1/\epsilon^3)$ in the 
unconstrained setting, i.e., $\CM = \RBB^d$.
In this paper, we show how to find an $\epsilon$-first order stationary point for the more general 
constrained problem \eqref{eqn_cost_to_optimize} with the same improved trajectory complexity in a 
projection free manner.
\blue{Moreover, we emphasize that the RL problem \eqref{eqn_cost_to_optimize} is strictly a special case of the more general objective \eqref{eqn_DR_submodular_maximization_loss}: The reward function $\RM(\tau)$ in \eqref{eqn_cost_to_optimize} depends only on the random variable $\tau$ and is independent of the variable $\theta$; In \eqref{eqn_DR_submodular_maximization_loss}, the function $\tilde{F}(\xB; \zB)$ depends on both the variable $\xB$ and the random variable $\zB$, which is hence more general.
}

\subsection{Related Work} \label{section_related_work}

\begin{table}[t]
	\scriptsize
\begin{center}
\begin{tabular}{|c|c|c|c|c|c|c|}
\hline
\textbf{Ref.}& \textbf{setting}  & \textbf{assumptions}& \textbf{batch} & \textbf{rate/iter} &\textbf{complexity}  & \textbf{non-obl.}\\
\hline
\cite{DBLP:conf/icml/Jaggi13} & det.& smooth & --- &$\mathcal{O}(1/{t})$ & ---  & \ding{55}\\ 
\hline
\cite{DBLP:conf/icml/HazanK12} & stoch.    &smooth, bounded grad. & $\mathcal{O} (t)$ &$\mathcal{O}(1/{t^{1/2}}) $ &$\mathcal{O}(1/ \epsilon^{4}) $ & \ding{55}\\
\hline
\cite{hazan2016variance} & stoch.    &smooth, bounded grad. & $\mathcal{O} (t^2)$ &$\mathcal{O}(1/{t}) $ &$\mathcal{O}(1/ \epsilon^{3}) $ & \ding{55}\\
\hline
\cite{mokhtari2018stochastic} & stoch.  & smooth, bounded var.& $\mathcal{O} (1)$ &$O(1/{t^{1/3}}) $&$\mathcal{O}(1/ \epsilon^{3}) $ & \ding{55}\\
\hline\hline
This paper & stoch.  & smooth, bounded var.& {{$\mathcal{O}(1/\epsilon)$}}  & {{$\mathcal{O}(1/\epsilon)$}}  &$\mathcal{O}(1/ \epsilon^{2}) $ & \checkmark\\
\hline
\end{tabular}
\caption{Convergence guarantees of conditional gradient (FW) methods for convex minimization}
\vspace{-8mm}
\end{center}
\end{table}

Our work on the conditional gradient method in the non-oblivious stochastic setting has consequences for convex, continuous submodular, and non-convex cases. In the following, we 
review 
some of the most relevant work and our results with respect to them. We would like to emphasize that all the other related work only provide guarantees for the oblivious setting.  

\vspace{-1mm}
\medskip\noindent{\bf Convex minimization.\ }The problem of minimizing a stochastic convex 
function 
subject to a convex constraint using stochastic projected gradient descent-type methods has been 
studied extensively in the past 
\cite{robbins1951stochastic,nemirovski1978cezari,nemirovskii1983problem}. Although 
stochastic 
gradient computation is inexpensive, the cost of projection step can be prohibitive 
\cite{fujishige2011submodular} or intractable \cite{collins2008exponentiated}. In such cases, 
the 
projection-free methods, a.k.a., Frank-Wolfe or conditional gradient, are the method of choice 
\cite{frank1956algorithm,DBLP:conf/icml/Jaggi13}. In the stochastic setting, the online 
Frank-Wolfe 
algorithm proposed in \cite{DBLP:conf/icml/HazanK12} requires $\mathcal{O}(1/\epsilon^4)$ 
stochastic 
gradient evaluations to reach an $\epsilon$-approximate optimum, i.e., ${F(\xB)-OPT}\leq 
\epsilon$, 
under the assumption that the objective function is convex and has bounded gradients. The 
stochastic 
variant of Frank-Wolfe studied in \cite{hazan2016variance}, uses an increasing batch size 
of 
$b=\mathcal{O}(t^2)$ (at iteration $t$) to obtain an improved stochastic oracle complexity of 
$\mathcal{O}(1/\epsilon^3)$ under the assumptions that the expected objective function is smooth  
and Lipschitz continuous. Recently, \cite{mokhtari2018stochastic} proposed a momentum gradient 
estimator which achieves a similar $\mathcal{O}(1/\epsilon^3)$ stochastic gradient 
evaluations while fixing the batch-size to~1. \cite{lan2016conditional} proposed a stochastic 
conditional gradient sliding method which finds an $\epsilon$-approximate solution after 
$\mathcal{O}(1/\epsilon^2)$ stochastic gradient evaluations and $\mathcal{O}(1/\epsilon)$ calls to 
a 
linear minimization oracle. The main idea in gradient sliding algorithms is to simulate projected 
gradient 
descent step by solving a sequence of properly chosen linear minimization problems 
\cite{lan2016conditional,lan2016gradient,lan2017conditional,braun2017lazifying}. Our proposed 
method \FWPP also requires $\mathcal{O}(1/\epsilon^2)$ calls to a stochastic 
gradient 
oracle (for oblivious and non-oblivious settings) and 
$\mathcal{O}(1/\epsilon)$ calls to a linear minimization oracle. However, unlike gradient sliding, we 
do 
not resort in simulating the projection step and more closely follow the recipe of the  
Frank-Wolfe 
method. In this sense, \FWPP might be considered the first variant of Frank-Wolfe 
that 
achieves the optimum convergence rate in the convex setting.

\vspace{-1mm}
\medskip\noindent{\textbf{Submodular maximization.}}
Submodular set functions \cite{nemhauser1978analysis}  capture the intuitive  notion of diminishing returns and
have become increasingly  important in various fields. 
The celebrated result of  \cite{nemhauser1978analysis} shows that for a monotone submodular function and subject to a  cardinality constraint, a simple greedy algorithm achieves the tight $(1-1/e)$ approximation 
guarantee. 
However, the vanilla greedy method does not provide the tightest guarantees for many classes of 
feasibility constraints. To circumvent this issue, the continuous relaxation of submodular functions, 
through the multilinear extension, have been extensively studied \cite{vondrak2008optimal, 
calinescu2011maximizing, chekuri2014submodular, feldman2011unified, gharan2011submodular, 
sviridenko2017optimal}. In particular, it is known that the continuous greedy algorithm achieves the 
tight $(1-1/e)$ approximation guarantee for monotone submodular functions under a general 
matroid constraint \cite{calinescu2011maximizing}. 
In the non-monotone setting, a slight variant of continuous greedy, called measured continuous greedy, achieves $1/e$ approximation 
guarantee \cite{feldman2011unified}.
\blue{In the absence of constraints, two recent work \cite{Niazadeh2020online,roughgarden2018optimal} are able to achieve the tight $1/2$ approximation guarantee for the online unconstrained non-monotone submodular maximization problem by exploiting the offline bi-greedy algorithm proposed in  \cite{buchbinder2018deterministic}.}
The continuous relaxation of submodular functions have also been used to robustify submodular optimization in the stochastic settings \cite{karimi17,hassani2017gradient,mokhtari2018conditional}.

\begin{table}[t]
	\footnotesize
\begin{center}
\begin{tabular}{| c| c| c| c| c| c| }
\hline
\textbf{Ref.}& \textbf{setting} & \textbf{function}  & \textbf{const.} & \textbf{utility} & \textbf{complexity}  \\
\hline 
 \cite{chekuri2015multiplicative} & det.& mon.smooth sub. & poly. & $(1-1/e) \rm{OPT}-\epsilon $ & $O(1/{\epsilon^{2}})$\\
\hline
\cite{bian2017guaranteed} & det. &  mon. DR-sub. &cvx-down & $(1-1/e) \rm{OPT}-\epsilon $ & $O(1/{\epsilon}) $\\
\hline
\cite{DBLP:conf/nips/BianL0B17} & det. &  non-mon. DR-sub. &cvx-down & $(1/{e}) \rm{OPT}-\epsilon $ & $O(1/{\epsilon})$\\
\hline
\cite{hassani2017gradient}  & det. &mon. DR-sub. & convex &$(1/2)\rm{OPT}-\epsilon$& $O(1/{\epsilon})$ \\
\hline
\cite{hassani2017gradient} & stoch. &mon. DR-sub. & convex &$(1/2)\rm{OPT}-\epsilon$& $O(1/{\epsilon^{2}})$ \\
\hline
\cite{mokhtari2018stochastic} & stoch. &mon. DR-sub. & convex &$(1-1/e)\rm{OPT}-\epsilon$& $O(1/{\epsilon^{3}})$ \\
\hline
\cite{mokhtari2018stochastic}  & stoch. &non-mon. DR-sub. & convex &$(1/e)\rm{OPT}-\epsilon$& $O(1/{\epsilon^{3}})$ \\
\hline
\hline
This paper & stoch. &mon. DR-sub. & convex &$(1-1/e)\rm{OPT}-\epsilon$& 
{\color{black}{$\mathcal{O}(1/\epsilon^2)$}} \\
\hline
This paper & stoch. &non-mon. DR-sub. & convex &$(1/e)\rm{OPT}-\epsilon$& {\color{black}{$\mathcal{O}(1/\epsilon^2)$}}  \\
\hline
\end{tabular}
\caption{Convergence guarantees for continuous DR-submodular function maximization}
\vspace{-8mm}
\end{center}
\end{table}

%

Continuous DR-submodular functions, an important subclass of non-convex functions, generalize 
the 
notion of diminishing returns to the continuous domains  
\cite{wolsey1982analysis, bach2015submodular}. 
%
It has been recently shown that monotone continuous 
DR-submodular functions  can be (approximately) maximized over convex bodies using first-order methods \cite{bian2017guaranteed, 
hassani2017gradient, mokhtari2018conditional}. When exact gradient information is available, 
\cite{bian2017guaranteed} showed 
that the continuous greedy algorithm, which itself is a variant of the conditional gradient method, 
achieves 
$[(1-1/e)\text{OPT}-\epsilon]$   with $O(1/\epsilon)$ gradient evaluations. However, the problem 
becomes considerably more challenging when we only have access to a \emph{stochastic} 
first-order 
oracle. In particular, \cite{hassani2017gradient} showed that the 
stochastic gradient ascent  achieves  $[(1/2)\OPT-\epsilon]$ by using 
$O(1/\epsilon^2)$ stochastic gradients. In contrast, \cite{mokhtari2018conditional} proposed the 
stochastic variant of the continuous greedy algorithm that achieves $[(1-1/e)\text{OPT}-\epsilon]$ 
by 
using $O(1/\epsilon^3)$  stochastic gradients. In this paper, we show that \SCGPP achieves 
$[(1-1/e)\text{OPT}-\epsilon]$ by  $O(1/\epsilon^2)$ stochastic gradient evaluations. We also 
show that 
 the convergence rate of $O(1/\epsilon^2)$ is optimal.  We further generalize our result to the 
 non-monotone DR-continuous submodular setting, by proposing the  stochastic variant of the 
 measured continuous greedy \cite{feldman2011unified}. 
 Specifically, \NMSCGPP  achieves a 
 $[(1/e)\OPT 
 -\epsilon]$ solution  by using $\OM({1}/{\epsilon^2})$ stochastic gradients. Note that for the non-monotone DR-submodular maximization (in contrast to the monotone case), one needs the extra assumption that the set $\CM$ is down-closed, or otherwise no constant factor approximation in polynomial time is possible.   
%
%
%

\begin{table}[t]
	\scriptsize
	\begin{center}
		\begin{tabular}{|c|c|c|c|c|c|c|}
			\hline
			\textbf{Ref.}& \textbf{setting}  & \textbf{assumptions}& \textbf{batch} & \textbf{\#iter} &\textbf{complexity}  & \textbf{non-obl.}\\
			\hline
			\cite{lacoste2016convergence} & det.& smooth & --- &$\mathcal{O}(1/{{\epsilon}^{2}})$ & ---  & \ding{55}	\\ 
			\hline
			\cite{hazan2016variance} & stoch.    &smooth, bounded var. & $\mathcal{O} (1/\epsilon^2)$ &$\mathcal{O}(1/{{\epsilon}^{2}}) $ &$\mathcal{O}(1/ \epsilon^{4}) $ & \ding{55}\\
			\hline
			\cite{hazan2016variance} & stoch.    &smooth, bounded var. & $\mathcal{O} (1/\epsilon^{4/3})$ &$\mathcal{O}(1/\epsilon^2) $ &$\mathcal{O}(1/ \epsilon^{10/3}) $ & \ding{55}\\
			\hline
			\blue{\cite{shen2019complexities,yurtsever2019conditional}} & stoch.  & smooth, bounded var.& {{$\mathcal{O}(1/\epsilon)$}}  & {{$\mathcal{O}(1/\epsilon^2)$}}  &$\mathcal{O}(1/ \epsilon^{3}) $ & \ding{55}\\
			\hline
			\cite{zhang2020one} & stoch.  & smooth, bounded var.& {{$\mathcal{O}(1)$}}  & {{$\mathcal{O}(1/\epsilon^2)$}}  &$\mathcal{O}(1/ \epsilon^{3}) $ & \checkmark\\
			\hline			
			\hline
			This paper & stoch.  & smooth, bounded var.& {{$\mathcal{O}(1/\epsilon)$}}  & {{$\mathcal{O}(1/\epsilon^2)$}}  &$\mathcal{O}(1/ \epsilon^{3}) $ & \checkmark	\\
			\hline
		\end{tabular}
		\caption{\blue{Convergence guarantees of conditional gradient (FW) methods for non-convex minimization}}
		\vspace{-6mm}
	\end{center}
\end{table}

\vspace{-1mm}
\medskip\noindent{\bf Nonconvex minimization.\ } The focus of this paper is on constrained 
optimization in the non-oblivious  stochastic setting. Nevertheless, convergence to  first-order 
stationary points (FOSP) for non-convex functions has been widely studied in the unconstrained case for oblivious problems 
\cite{DBLP:conf/cdc/ReddiSPS16,reddi2016stochastic,DBLP:conf/icml/ZhuH16,
	DBLP:conf/nips/LeiJCJ17}. Recently, 
the finite-time analysis for convergence to an FOSP of 
\textit{constrained} problems has also 
received a lot of attention. In particular, in the deterministic setting, 
\cite{lacoste2016convergence} 
showed that the sequence of iterates generated by the update of Frank-Wolfe converges 
to an 
$\epsilon$-FOSP after $\mathcal{O}(1/\epsilon^{2})$ iterations. In contrast, 
\cite{ghadimi2016mini} considered the norm of gradient mapping as a measure of non-stationarity 
and 
showed that the projected gradient method has the same complexity of 
$\mathcal{O}(1/\epsilon^{2})$. 
Similar results for the accelerated projected gradient method was shown in 
\cite{ghadimi2016accelerated}. Note that $\epsilon$-FOSP and the norm of gradient mapping are 
not directly related to one another. Adaptive cubic regularization methods in
\cite{cartis2012adaptive,cartis2013evaluation,cartis2015evaluation} improved these results 
by using second-order information in order to obtain an $\epsilon$-FOSP after 
$\mathcal{O}(1/\epsilon^{3/2})$ iterations. Later, \cite{mokhtari2018escaping} showed that 
projected gradient descent reaches an $\epsilon$-FOSP after $\mathcal{O}(1/\epsilon^2)$ 
iterations in deterministic setting in terms of the Frank-Wolfe gap. In the oblivious 
stochastic
setting, \cite{reddi2016stochastic} introduced a stochastic variant of  Frank-Wolfe which finds an 
$\epsilon$-FOSP after $\mathcal{O}(1/\epsilon^4)$ stochastic gradient evaluations and 
$\mathcal{O}(1/\epsilon^2)$ calls to a linear minimization oracle. \cite{DBLP:conf/icml/QuLX18} 
introduced a variant of the gradient sliding method that finds an $\epsilon$-FOSP after 
$\mathcal{O}(1/\epsilon^2)$ stochastic gradient evaluations and $\mathcal{O}(1/\epsilon^2)$ calls 
to a linear minimization oracle when we measure first-order optimality in terms of proximal gradient 
mapping. Again, this result can not be compared with ours as we measure first-order 
optimality based on the the Frank-Wolfe gap. 

\vspace{-1mm}
\medskip\noindent{\bf Concurrent Work.\ }  In this part, we briefly discuss some  recent results from 
concurrent  works that appeared after we made the first version of this paper publicly available on 
arXiv. In particular, \cite{shen2019complexities, yurtsever2019conditional} considered the 
oblivious stochastic setting (a special case of 
problem~\eqref{eqn_DR_submodular_maximization_loss}) and introduced a variance reduced 
version of 
Frank-Wolfe 
method based on the Stochastic Path-Integrated Differential Estimator (SPIDER) approach 
\cite{fang2018spider}.  Their proposed methods find an $\epsilon$-FOSP in 
non-convex minimization after $\mathcal{O}(1/\epsilon^3)$ stochastic gradient evaluations and 
$\mathcal{O}(1/\epsilon^2)$ calls to a linear minimization oracle. \cite{yurtsever2019conditional} 
also noted that in the oblivious stochastic  convex minimization, the same method achieves an 
$\epsilon$-approximate 
solution after $\mathcal{O}(1/\epsilon^2)$ stochastic gradient evaluations and 
$\mathcal{O}(1/\epsilon)$ calls to a linear minimization oracle. Finally, \cite{zhang2019quantized} 
proposed a quantized Frank-Wolfe algorithm, by relying on  SPIDER, to develop a 
communication-efficient distributed 
method.

\section{Preliminaries} \label{section_preliminaries}
In this section, we state some of the required definitions and then review variance reduced methods for stochastic optimization. 

\begin{definition} \label{def:lips_grad}
A function $\phi: \RBB^n\to\RBB$ is $L$-smooth if it has $L$-Lipschitz continuous gradients on $\RBB^n$, i.e., for any $\xB, \hat{\xB} \in \RBB^n$, we have
$
|| \nabla \phi(\xB) - \nabla \phi(\hat{\xB}) || \leq L ||\xB - \hat{\xB}||
$.
\end{definition}

%

\begin{definition}  \label{def:strong_convex}
A function $\phi: \RBB^n\to\RBB$ is convex on $\mathbb{R}^n$ if we have
$
\phi(\hat{\xB}) \geq \phi(\xB) + \nabla \phi(\xB) ^T (\hat{\xB} - \xB)
$ for any $\xB, \hat{\xB} \in \RBB^n$. Further, $\phi(\xB)$ is concave if $-\phi(\xB)$ is convex.
\end{definition}

\medskip\noindent{\bf Submodularity.} 
A set function $f:2^V\rightarrow \RBB_+$, defined on the ground set $V$,  is submodular if $f(A)+f(B)\geq f(A\cap B) + f(A\cup B),$ for all subsets $A,B\subseteq V$. Even though submodularity is mostly considered on discrete domains, the notion can
be naturally extended to arbitrary lattices \cite{fujishige91}. To this aim, let us consider a subset of $\RBB_+^d$ of the form $\XM = \prod_{i=1}^d \XM_i$ where each $\XM_i$ is a compact subset of $\RBB_+$. A function $F:\XM\rightarrow \RBB_+$ is  \textit{continuous submodular} if for all $(\xB,\yB)\in \XM \times \XM$, we have 
$F(\xB)+ F(\yB) \geq F(\xB\vee \yB) + F(x\wedge \yB)$, 
where $\xB\vee \yB \doteq\max(\xB,\yB)$ (component-wise) and $\xB\wedge \yB \doteq \min (\xB,\yB)$ (component-wise). A submodular function is monotone if for any $\xB,\yB\in\XM$ such that $\xB\leq \yB$, we have $F(\xB)\leq F(\yB)$ (here, by $\xB \leq \yB$ we mean that every coordinate of $\xB$ is less than the corresponding coordinate of $\yB$). 
When  twice differentiable, $F$ is submodular if and only if all cross-second-derivatives are non-positive \cite{bach2015submodular}, i.e., 
$\forall i\neq j, \forall \xB\in \XM, ~~ \frac{\partial^2 F(\xB)}{\partial x_i \partial x_j} \leq 0$.
This expression makes it clear that continuous submodular functions are not convex nor concave in general, as concavity (convexity) implies $\nabla^2 F\preceq 0$ ($\nabla^2 F \succeq 0$). 
A proper subclass of submodular functions are called \textit{DR-submodular} \cite{bian2017guaranteed, soma2015generalization} if for all $\xB,\yB\in\XM$ such that $\xB\leq \yB$  and any standard basis vector $\eB_i\in\RBB^n$ and a non-negative number $z\in\RBB_+$ such that $z\eB_i+\xB\in\XM$ and $z\eB_i+\yB\in\XM$, then,
$F(z\eB_i+\xB)-F(\xB)\geq F(z\eB_i+\yB)-F(\yB). $
One can easily verify that for a differentiable DR-submodular function the gradient is an antitone mapping, i.e., for all $\xB,\yB \in \XM$ such that $\xB\leq \yB$ we have  $\nabla F(\xB) \geq \nabla F(\yB)$ \cite{bian2017guaranteed}. A crucial example of a DR-submodular function is the multilinear extension \cite{calinescu2011maximizing} that we study in Section~\ref{discrete}.

\medskip\noindent{\bf Variance Reduction.\ }  
Beyond the vanilla stochastic gradient, variance reduced algorithms \cite{schmidt2017minimizing,johnson2013accelerating,defazio2014saga,reddi2016stochastic,nguyen2017sarah,allen2018natasha} have been successful in reducing  stochastic first-order oracle complexity in the \emph{oblivious} stochastic optimization
\vspace{-1mm}
\begin{equation}\label{eqn_oblivious_loss}
	\max_{\xB \in \CM} F(\xB) \ :=\ \max_{\xB \in \CM}\ \EBB_{\zB\sim p(\zB)} [\tilde F(\xB; \zB)],
	\vspace{-1mm}
\end{equation}
where each function $\tilde F(\cdot; \zB)$ is \emph{$L$-smooth}.
In contrast to (\ref{eqn_DR_submodular_maximization_loss}), the underlying distribution $p$ of (\ref{eqn_oblivious_loss}) is invariant to the variable $\xB$ and is hence called oblivious. 
We will now explain a recent variance reduction technique for solving \eqref{eqn_oblivious_loss} using stochastic gradient information. 
Consider the following \emph{unbiased} estimate of the gradient at $\xB^{t}$:
\vspace{-1mm}
\begin{equation}\label{eqn_VR_estimator}
\gB^{t} \defi {\gB}^{t-1} + \nabla \tilde F(\xB^{t};\MM) - \nabla \tilde F({\xB}^{t-1};\MM),\vspace{-1mm}
\end{equation}
where $\nabla \tilde F(\yB;\MM) \defi \frac{1}{|\MM|} \sum_{\zB\in\MM} \nabla \tilde F(\yB; \zB)$ for some $\yB \in \RBB^{d}$, $\gB^{t-1}$ is an unbiased gradient estimator at $\xB^{t-1}$, and $\MM$ is a mini-batch of random samples drawn from $p(\zB)$.
\cite{fang2018spider} showed that, with the gradient estimator (\ref{eqn_VR_estimator}), $\OM({1}/{\epsilon^3})$ stochastic gradient evaluations are sufficient to find an $\epsilon$-first-order stationary point of Problem (\ref{eqn_oblivious_loss}), improving upon the $\OM({1}/{\epsilon^4})$ complexity of SGD. A crucial property leading to the success of the variance reduction method given in \eqref{eqn_VR_estimator} is that $\nabla \tilde F(\xB^{t};\MM)$ and $\nabla \tilde F({\xB}^{t-1};\MM)$ use \emph{the same} minibatch sample $\MM$ in order to exploit the $L$-smoothness of component functions $\tilde F(\cdot;\zB)$.
Such construction is only possible in the oblivious setting where $p(\zB)$ is independent of the choice of $\xB$, and would introduce bias in the more general non-oblivious case (\ref{eqn_DR_submodular_maximization_loss}): To see this point, let $\MM$ be the minibatch of random variable $\zB$ sampled according to distribution $p(\zB;\xB^{t})$.
We have $\EBB[ \nabla \tilde F(\xB^{t};\MM)] = \nabla F(\xB^{t})$ but $\EBB[ \nabla \tilde F({\xB}^{t-1};\MM)] \neq \nabla F({\xB}^{t-1})$ since the distribution $p(\zB;\xB^{t-1})$ is not the same as $p(\zB;\xB^{t})$.
The same argument renders all the existing variance reduction techniques inapplicable to the non-oblivious setting of Problem (\ref{eqn_DR_submodular_maximization_loss}).

\section{Stochastic (Non-)Convex Minimization}\label{sec:convex}

In this section, we focus on two specific cases of Problem \eqref{eqn_DR_submodular_maximization_loss} where the objective function $F$ is a concave function or a general nonconcave function. As stated earlier, maximizing a (non-)concave function can be written as minimizing a (non-)convex function. We hence rewrite \eqref{eqn_DR_submodular_maximization_loss} as 
\begingroup
\setlength\abovedisplayskip{5pt}
\begin{equation}\label{eqn_convex_minimization_loss}
	\min_{\xB \in \CM} F(\xB) \ :=\ \min_{\xB \in \CM}\ \EBB_{\zB\sim p(\zB; \xB)} [\tilde{F}(\xB; \zB)],\vspace{-1mm}
\end{equation} 
\endgroup
where we assume that the expected function $F$ is either convex or a general nonconvex function. In this section we study a non-oblivious stochastic optimization problem, but our results trivially hold for the oblivious stochastic problem as a special case. 

\subsection{Stochastic Frank-Wolfe++}\label{sec:sfw}
Now we introduce the {\texttt{Stochastic} \texttt{Frank} \texttt{-Wolfe}{++}\xspace} method (\FWPP) to solve the non-oblivious minimization problem~\eqref{eqn_convex_minimization_loss}. Recall that the Frank-Wolfe method requires access to the gradient of the objective function $\nabla F$. However, evaluation of $\nabla F$ may not be possible in many settings as either the probability distribution $p$ is not available or evaluating the expectation in \eqref{eqn_convex_minimization_loss} is computationally prohibitive. Our goal is to design an unbiased estimator $\gB$ for approximating the exact gradient $\nabla F$ that is computationally affordable and has a low variance. Once the gradient approximation $\gB^t$ for step $t$ is evaluated we can find the descent direction $\vB^t$ by solving the linear optimization program 
\begingroup
\setlength\abovedisplayskip{5pt}
\begin{equation}
\vB^t := \argmin_{\vB\in\CM}\{\vB^\top\gB^t\},
\end{equation}
\endgroup
and then compute the updated variable $\xB^{t+1}$ by performing the following update 
\begingroup
\setlength\abovedisplayskip{5pt}
\begin{equation} \label{eqn_update_fw}
\xB^{t+1} := \xB^t + \eta_{t}\ (\vB^t - \xB^t),\vspace{-1mm}
\end{equation}
\endgroup
where $\eta_t$ is a properly chosen stepsize. The steps of the \FWPP are summarized in Algorithm~\ref{algo_FW++}. In \FWPP we restart the gradient estimation $\gB$ after an episode of iterates with a certain length. This step is necessary to ensure that the noise of gradient approximation stays bounded by a proper constant. The details for computing the gradient approximation $\gB^t$ is provided in the following section. 
\begin{algorithm}[t]
	\caption{\sfw(\FWPP)}
	\label{algo_FW++}
	\begin{algorithmic}[1]
		\REQUIRE Number of iteration $T$, minibatch size $|\MM_0^t|$ and $|\MM_h^t|$, step size $\eta_{t}$
		\FOR{$t = 0$ \TO $T$}
		\IF{$\mod(t, q) = 0$ (non-concave) or $\log_2 t \in \ZBB$ (concave)}
		\STATE Sample $\MM^t_0$ of $\zB$ with distribution $p(\zB; \xB^t)$ to find $
		\gB^t \defi \nabla F(\xB^t; \MM_0^t)$;
		\ELSE
		\STATE Sample a minibatch $\MM_h^t$ of $(a, \zB(a))$ to calculate $\tilde{\nabla}_{t}^2$ using (\ref{eqn_integral_estimator}) 
		\STATE Compute $\gB^t := \gB^{t-1} + \tilde{\nabla}_{t}^2(\xB^{t} - \xB^{t-1})$; \qquad$\diamond$ Or use \eqref{eqn_approximate_matrix_vector_product} instead.
		\ENDIF
		\STATE $\vB^t := \argmax_{\vB\in\CM}\{\vB^\top\gB^t\}$;
		\STATE $\xB^{t+1} := \xB^t + \eta_{t}\cdot(\vB^t - \xB^t)$;
		\ENDFOR
		\ENSURE $\xB^{\bar{t}}$ with $\bar{t}$ uniformly sampled from $[T]$ (non-concave case); $\xB^{T}$ (concave case).
	\end{algorithmic}
\end{algorithm}
\subsubsection{Stochastic gradient approximation} \label{section_stochastic_gradient_approximation}
Given a sequence of iterates, $\{\xB^{s} \}_{s=0}^t$, the gradient of $F$ at $\xB^{t}$ can be written in a path-integral form as follows
\begingroup
\setlength\abovedisplayskip{5pt}
\begin{equation}
	\nabla F(\xB^{t}) = \nabla F(\xB^{0}) + \sum_{s=1}^{t} \left\{\Delta^s \defi \nabla F(\xB^{s}) - \nabla F(\xB^{s-1})\right\}.\vspace{-1mm}
\end{equation}
\endgroup
By obtaining an unbiased estimate of $\Delta^{t} = \nabla F(\xB^{t}) - \nabla F(\xB^{t-1})$, and reusing the previous unbiased estimates for $s<t$, we  obtain recursively an unbiased estimator of $\nabla F(\xB^{t})$ which has a reduced variance.
Estimating $\nabla F(\xB^{s})$ and $\nabla F(\xB^{s-1})$ separately as suggested in  (\ref{eqn_VR_estimator}) would cause the bias issue in the the non-oblivious case (see the discussion at the end of section \ref{section_preliminaries}). Therefore, we propose an approach for \emph{directly estimating the difference} $\Delta^t = \nabla F(\xB^{t}) - \nabla F({\xB}^{t-1})$ in an unbiased manner.

We construct an unbiased estimator $\gB^t$ of the gradient $\nabla F(\xB^t)$ by adding an unbiased estimate $\tilde{\Delta}^t$ of the gradient difference $\Delta^t = \nabla F(\xB^{t}) - \nabla F(\xB^{t-1})$ to $\gB^{t-1}$, where $\gB^{t-1}$ is an unbiased estimator of $\nabla F(\xB^{t-1})$. Note that $\Delta^t$ can be written as 
\begin{equation}
	\Delta^t = \int_{0}^{1} \nabla^2 F(\xB(a))(\xB^{t} - \xB^{t-1})\dB a  = \left[\int_{0}^{1} \nabla^2 F(\xB(a))\dB a \right] (\xB^{t} - \xB^{t-1}),
	\label{eqn_integral}
\end{equation}
where $\xB(a) \defi a \cdot\xB^{t}+(1-a)\cdot \xB^{t-1}$ for $a \in[0, 1]$.
Hence, if we sample the parameter ${a}$ uniformly at random from the interval $[0, 1]$, it can be easily verified that $\tilde{\Delta}^t:=\nabla^2 F(\xB({a})) (\xB^{t} - \xB^{t-1})$ is an unbiased estimator of the gradient difference ${\Delta}^t$ since
\begin{equation}
\EBB_{{a}} [\nabla^2 F(\xB({a})) (\xB^{t} - \xB^{t-1})] = \nabla F(\xB^{t}) - \nabla F(\xB^{t-1}).
\label{eqn_integral_2}
\end{equation}
Hence, all we need is an unbiased estimator of the Hessian-vector product $\nabla^2 F(\yB) (\xB^{t} - \xB^{t-1})$ for the non-oblivious objective $F$ at an arbitrary $\yB\in\CM$. Next, we present an unbiased estimator of $\nabla^2 F(\yB)$ for any $\yB\in\CM$ that can be evaluated efficiently.


\begin{lemma} \label{lemma_hessian_estimator}
	Consider $\yB \in \CM$, and $\zB$ with distribution $p(\zB;\yB)$ and define
	\vspace{-1mm}
	\begin{equation}\label{eqn_unbiased_second_order_differential_estimator}
		\begin{aligned}
			\tilde{\nabla}^2 F(\yB; \zB) \defi &\tilde{F}(\yB;\zB)[\nabla\log p(\zB;\yB)] [\nabla\log p(\zB;\yB)]^\top + [\nabla \tilde{F}(\xB;\zB)][\nabla\log p(\zB;\yB)]^\top \\
			& + [\nabla\log p(\zB;\yB)][\nabla \tilde{F}(\yB;\zB)]^\top + \nabla^2 \tilde{F}(\yB;\zB) + \tilde{F}(\yB; \zB)\nabla^2\log p(\zB;\yB).
		\end{aligned}
		\vspace{-1mm}
	\end{equation}
	Then, $\tilde{\nabla}^2 F(\yB; \zB)$ is an unbiased estimator of $\nabla^2 F(\yB)$.
\end{lemma}

The result in Lemma~\ref{lemma_hessian_estimator} shows how to evaluate an unbiased estimator of the Hessian $\nabla^2 F(\yB)$. 
If we consider $a$ as a random variable with a uniform distribution over the interval $[0, 1]$, then we can define the random variable $\zB(a)$ with the probability distribution $p(\zB(a); \xB(a))$ where $\xB(a) $ is defined as $\xB(a) := a\cdot\xB^{t}+(1-a)\cdot \xB^{t-1}$.
Considering these two random variables and the result of Lemma \ref{lemma_hessian_estimator}, we can construct an unbiased estimator of the integral $\int_{0}^{1} \nabla^2 F(\xB(a))\dB a$ in (\ref{eqn_integral}) by 
\vspace{-1mm}
\begin{equation}\label{eqn_integral_estimator}
	\tilde{\nabla}_{t}^2 \defi \frac{1}{|\MM|}\sum_{(a, \zB(a)) \in \MM} \tilde{\nabla}^2 F(\xB(a); \zB(a)),\vspace{-1mm}
\end{equation}
where $\MM$ is a minibatch containing $|\MM|$ samples of random tuple $(a, \zB(a))$. 
%
Once we construct $\tilde{\nabla}_{t}^2$, the gradient difference $\Delta^t$ can be approximated by 
\begin{equation} \label{eqn_gradient_difference_estimator}
	\tilde{\Delta}^t := \tilde{\nabla}_{t}^2 (\xB^{t} - \xB^{t-1}). 
\end{equation}
Note that for the general objective $F(\cdot)$, the matrix-vector product $\tilde{\nabla}_{t}^2(\xB^{t} - \xB^{t-1})$ requires $\OM(d^2)$ computation and memory. To resolve this issue, in Section~\ref{sec:hessian_product}, we provide an implementation of \eqref{eqn_gradient_difference_estimator} using only first-order information which reduces the computational and memory complexity to $\OM(d)$.
Using $\tilde{\Delta}^t$ as an unbiased estimator for the gradient difference $\Delta^t$, we can define our gradient estimator as
\vspace{-.2cm}
\begin{equation}\label{g_def_total}
\gB^{t} = \nabla \tilde{F}(\xB^0; \MM_0) + \sum_{i=1}^{t}\tilde{\Delta}^t. \vspace{-.2cm}
\end{equation}
Indeed, this update can also be rewritten in a recursive way as 
\vspace{-.1cm}
\begin{equation}\label{g_def_rec}
\gB^{t} =  \gB^{t-1}+\tilde{\Delta}^t,\vspace{-.1cm}
\end{equation}
once we set $\gB^{0}=\nabla \tilde{F}(\xB^0; \MM_0)$. Note that the proposed approach for gradient approximation in \eqref{g_def_total} has a variance reduction mechanism which leads to the optimal computational complexity of \FWPP $\,$ in terms of number of calls to the stochastic oracle. We further highlight this point in the convergence analysis of \FWPP.

\subsubsection{Implementation of the Hessian-Vector Product} \label{sec:hessian_product}

In this section, we focus on the computation of the gradient difference approximation $\tilde{\Delta}^t$ introduced in \eqref{eqn_gradient_difference_estimator}. We aim to come up with a scheme that avoids explicitly computing the matrix estimator $\tilde{\nabla}^{2}_t$ which has a complexity of $\OM(d^2)$, and present an approach that directly approximates $\tilde{\Delta}^t$ while only using the finite differences of gradients with a complexity of $\OM(d)$. Recall the definition of the Hessian approximation $\tilde{\nabla}^2_t$ in (\ref{eqn_integral_estimator}).
Computing $\tilde{\nabla}^2_t(\xB^{t} - \xB^{t-1})$ is equivalent to computing $|\MM|$ instances of $\tilde{\nabla}^2 F(\yB; \zB)(\xB^{t} - \xB^{t-1})$ for some $\yB \in \CM$ and $\zB \in \ZM$. Denote $\dB = \xB^{t} - \xB^{t-1}$ and use the expression in \eqref{eqn_unbiased_second_order_differential_estimator} to write
\begin{equation}
	\begin{aligned}
		\tilde{\nabla}^2 F(\yB; &\zB)\dB = \tilde{F}(\yB;\zB)[\nabla\log p(\zB;\yB)^\top\dB] \nabla\log p(\zB;\yB) + [\nabla\log p(\zB;\yB)^\top\dB]\nabla \tilde{F}(\xB;\zB) \\
		& + [\nabla \tilde{F}(\yB;\zB)^\top\dB][\nabla\log p(\zB;\yB)] + \nabla^2 \tilde{F}(\yB;\zB) \dB + \tilde{F}(\yB;\zB)\nabla^2\log p(\zB;\yB)\dB.
	\end{aligned}
	\label{eqn_hessian_vector_product}
\end{equation}
The first three terms can be computed\footnote{Also, note that pairs $(\xB,\yB)$ with $p(\xB,\yB) = 0$ will never be sampled (i.e. they are of measure zero) and hence do not cause any computational issues.}in time $\OM(d)$ and only the last two terms involve $\OM(d^2)$ operations, which can be approximated by the following finite gradient difference scheme.
For any twice differentiable function $\psi:\RBB^d \rightarrow \RBB$ and any $\dB \in \RBB^d$ with bounded norm $\|\dB\|\leq D$, we compute, for some small $\delta>0$,
\vspace{-1mm}
\begin{equation}\label{eqn_finite_gradient_difference}
	\phi(\delta; \psi)\defi \frac{\nabla \psi(\yB + \delta\cdot\dB) - \nabla \psi(\yB - \delta\cdot\dB)}{2\delta}\simeq \nabla^2 \psi(\yB) \dB.
	\vspace{-1mm}
\end{equation}
By considering the second-order smoothness of the function $\psi(\cdot)$ with constant $L_2$ we can show that for arbitrary $\xB, \yB \in \RBB^d$ it holds $\|\nabla^2 \psi(\xB) - \nabla^2 \psi(\yB)\|\leq L_2\|\xB -\yB\|$. Therefore, the error of the above approximation can be bounded by
\begin{equation}
\|\nabla^2 \psi(\yB) \dB - \phi(\delta; \psi)\| = \|\nabla^2 \psi(\yB) \dB - \nabla^2 \psi(\tilde{\xB}) \dB\| \leq D^2L_2\delta,
\label{eqn_error}
\end{equation}
where $\tilde{\xB}$ is obtained from the mean value theorem.
This quantity can be made arbitrary small by decreasing $\delta$ (up to the machine accuracy). 
By applying (\ref{eqn_finite_gradient_difference}) to functions $\psi(\yB) = \tilde{F}(\yB; \zB)$ and $\psi(\yB) = \log p(\zB;\yB)$, we can approximate (\ref{eqn_hessian_vector_product}) in $\OM(d)$: 
\begin{equation}
	\begin{aligned}
		\xi_\delta(\yB;\zB) \defi&\tilde{F}(\yB;\zB)[\nabla\log p(\zB;\yB)^\top\dB] \nabla\log p(\zB;\yB) + [\nabla\log p(\zB;\yB)^\top\dB]\nabla \tilde{F}(\xB;\zB) \\
		& + [\nabla \tilde{F}(\yB;\zB)^\top\dB][\nabla\log p(\zB;\yB)] + \phi(\delta; \tilde{F}(\yB;\zB)) + \phi(\delta; \log p(\zB;\yB)).
	\end{aligned}
\end{equation}
We further can define a minibatch version of this implementation as 
\begingroup
\setlength\abovedisplayskip{5pt}
\begin{equation}\label{eqn_approximate_matrix_vector_product}
	\xi_\delta(\xB;\MM) \defi \frac{1}{|\MM|}\sum_{(a, \zB(a))\in\MM} \xi_\delta(\xB(a);\zB(a)),\vspace{-1mm}
\end{equation}
\endgroup
which is used in Option II of Step 8 in Algorithm \ref{algo_FW++}.
Note that $\lim_{\delta\rightarrow 0} \xi_\delta(\xB;\MM) = \tilde{\Delta}^t$ and hence \eqref{eqn_gradient_difference_estimator} is a special case of \eqref{eqn_approximate_matrix_vector_product} by taking $\delta \rightarrow 0$.
Additionally, we show in later sections that setting $\delta = \OM(\epsilon^2)$ is sufficient, where $\epsilon$ is the target accuracy.
%

\subsection{Convergence Analysis of \FWPP: Nonconvex Setting}
In this section, we focus on solving Problem \eqref{eqn_convex_minimization_loss} when $F$ is smooth but nonconvex. In this case, our goal is to find an $\epsilon$-First-Order Stationary Point (FOSP), formally define as
\begingroup
\setlength\abovedisplayskip{5pt}
\begin{equation}\label{eqn_wolfe_gap_goal}
V_{\CM}(\xB; F) = \max_{\uB\in \CM} \langle \nabla F(\xB), \uB - \xB\rangle \leq \epsilon;\vspace{-1mm}
\end{equation}
\endgroup
where the parameter $V_{\CM}(\xB; F)$ captures distance to an FOSP and it is $0$ when $\xB$ is an FOSP. The parameter $V_{\CM}(\xB; F)$ is also known as Frank-Wolfe gap \cite{lacoste2016convergence}.

Before stating our main theorem for the general nonconvex case, we first formally state the required assumption for proving our results.

\begin{assumption}\label{ass_upper_bound_stochastic_function_value}
	The stochastic function $\tilde{F}(\xB;\zB)$ has bounded value for all $\zB \in \ZM$ and $\xB \in \CM$, i.e., $\exists B$ s.t. $\max_{\zB \in \ZM, \xB \in \CM} \tilde{F}(\xB;\zB) \leq B$.
\end{assumption}
\begin{assumption}\label{ass_compact_domain}
	The set $\CM$ is compact with diameter $D$.
\end{assumption}
\begin{assumption}	\label{ass_variance}
	Stochastic gradient  $\nabla \tilde{F}$ has bounded norm: $\forall \zB\!\in\!\ZM, \!\|\nabla \tilde{F}(\xB; \zB)\| \leq G_{\tilde{F}} $,
	and the norm of the gradient of  $\log p$ has bounded fourth-order moment, i.e., $
		\EBB_{\zB \sim p(\xB; \zB)} \|\nabla \log p(\zB;\xB)\|^4 \leq G_p^4$. 
	Further, we define $G = \max\{G_{\tilde{F}}, G_{p} \}.$
\end{assumption}
\begin{assumption}\label{ass_smooth}
	For all $\xB\in\CM$, the stochastic Hessian of $\nabla^2 \tilde{F}$ has bounded spectral norm $
		\forall \zB\in\ZM, \|\nabla^2 \tilde{F}(\xB; \zB)\| \leq L_{\tilde{F}},$
	and the spectral norm of the Hessian of the log-probability function has bounded second order moment:
$
	\EBB_{\zB \sim p(\zB; \xB)}[ \|\nabla^2 \log p(\zB; \xB)\|^2] \leq L_p^2
	$.
	Further, we define $
		L = \max\{L_{\tilde{F}}, L_{p} \}.
	$
\end{assumption}
\begin{assumption}\label{ass_Hessian_continuous}
	The stochastic Hessian is $L_{2, f}$-Lipschitz continuous, i.e, $\forall \xB, \yB\in\CM$ and all $\zB\in\ZM$,
	$	\|\nabla^2 \tilde{F}(\xB; \zB) - \nabla^2 \tilde{F}(\yB; \zB)\| \leq L_{2,\tilde{F}} \|\xB - \yB\|$. The Hessian of the log probability $\log p(\xB;\zB)$ is $L_{2,p}$-Lipschitz continuous: $\forall \xB, \yB\in\CM$ and all $\zB\in\ZM$, i.e.,
$	\|\nabla^2 \log p(\xB; \zB) - \nabla^2 \log p(\yB; \zB)\| \leq L_{2,p} \|\xB - \yB\|$.
	Also, we define $
		L_2 = \max\{L_{2,\tilde{F}}, L_{2,p} \}.$
\end{assumption}
\begin{remark}
	We note that high-order smoothness leads to faster convergence rates for gradient descent type algorithm.
	Concretely, for \emph{unconstrained} nonconvex problems, it is known that by assuming higher order smoothness (e.g. the Lipschitz continuity of the Hessian), one can obtain a faster convergence rate by exploiting only the first-order gradient algorithmically, e.g. \cite{Carmon2017} for the deterministic case and \cite{fang2019} for the SGD case (whether SPIDER has better rate with Lipschitz continuous Hessian is not known). However, in these results, the convergence rate \emph{explicitly} depends on the higher order smoothness parameter. 
	In the constrained case, whether higher order smoothness helps remains unknown.
	On the other hand, in our results, the higher order smoothness parameters $L_{2,\tilde{F}}$ and $L_{2,p}$ \emph{do not} enter the convergence results and are only required to evaluate the Hessian vector product when we have finite machine accuracy. 
	In fact, if we have the exact expressions of $\tilde F$ and $\log p$,  Assumption \ref{ass_Hessian_continuous} can be avoided by using the auto differential mechanism of Pytorch.  
\end{remark}

Next, we formally bound the variance of gradient approximation for \FWPP. 

\begin{lemma}
	Consider the \FWPP method outlined in Algorithm \ref{algo_FW++} and assume that in Step 8 we follow the update in \eqref{eqn_approximate_matrix_vector_product} to construct the gradient difference approximation $\tilde{\Delta}^t$ (Option II). If Assumptions \ref{ass_upper_bound_stochastic_function_value}, \ref{ass_compact_domain}, \ref{ass_variance}, \ref{ass_smooth}, and \ref{ass_Hessian_continuous} hold and we set the minibatch sizes to $|\MM_0| = ({G^2}/({\bar{L}^2D^2\epsilon^2}))$ and $|\MM| = {2}/{\epsilon}$, and the error of Hessian-vector product  approximation $\delta$ is $\OM(\epsilon^2)$ as in (\ref{eqn_small_delta}), then
	\begin{equation}
	\quad \EBB \left[ \|\gB^{t} - \nabla F(\xB^{t})\|^2 \right] \leq (1+\epsilon t)\bar{L}^2D^2\epsilon^2,\quad  \forall t\in \{0, \ldots, T-1 \},
	\end{equation}
where $\bar{L}$ is a constant defined as $\bar{L}^2 \defi 4B^2G^4 + 16G^4  + 4L^2 + 4B^2L^2$.
	\label{lemma_variance_general}
\end{lemma}

The result in Lemma \ref{lemma_variance_general} shows that by $|\MM| = \mathcal{O}(\epsilon^{-1})$ calls to the stochastic oracle at each iteration, the variance of gradient approximation in \FWPP after $t$ iterations is on the order of $\mathcal{O}((1+\epsilon t)\epsilon)$.  In the following theorem, we use this result to characterize the convergence properties of our proposed \FWPP method for solving stochastic non-convex minimization problems. 
For simplicity, we analyze the convergence of gradient-difference estimator \eqref{eqn_gradient_difference_estimator}.
However, similar results can be obtained for the Hessian-vector product estimator~\eqref{eqn_approximate_matrix_vector_product} by setting $\delta = \OM(\epsilon^2)$.

\begin{theorem}	\label{thm_nonconvex}
	Consider Problem \eqref{eqn_convex_minimization_loss} when $F$ is a general non-convex function. Further, recall the \FWPP method outlined in Algorithm \ref{algo_FW++}. Suppose the conditions in Assumptions \ref{ass_upper_bound_stochastic_function_value}, \ref{ass_compact_domain}, \ref{ass_variance}, \ref{ass_smooth},  and \ref{ass_Hessian_continuous} are satisfied. Further, let $\bar{L}^2 \defi 4B^2G^4 + 16G^4  + 4L^2 + 4B^2L^2$. If we set \FWPP parameters to $\eta_t = \epsilon/(\bar{L} D)$,  $|\MM_h^t| = 2G/\epsilon$, $q = \lceil G/(16\epsilon) \rceil$, and $|\MM_0^{t}| = G^2/(8\epsilon^2)$, then the iterates generated by \FWPP satisfy the condition
	$	\mathbb{E}\left[V_\CM(\xB^{\bar{t}}; F)\right]\leq 5\epsilon D$,
	where the total number of iterations is $T = \bar{L}(F(\xB^*) - F(\xB^0))/\epsilon^2$.
\end{theorem}

This theorem shows that after at most $\mathcal{O}(1/\epsilon^2)$ iterations, \FWPP reaches an $\epsilon$-FOSP. To characterize the overall complexity, we take into account the number of stochastic gradient evaluations per iteration, in the following corollary.
\begin{corollary}[{oracle complexity for non-concave case}] \label{col_nonconcave}
	Assume that the target accuracy $\epsilon$ satisfies $mod(T, q) = 0$.
	The overall stochastic complexity is 
	\vspace{-1.5mm}
	\begin{equation}
	\sum_{i=0}^T |\MM_h^i| + \sum_{k=0}^{T/q} |\MM_0^{qk}| \!=\! \OM\left(\bar{L}G(F(\xB^*) - F(\xB^0))/\epsilon^3\right).\vspace{-1.5mm}
	\end{equation}
\end{corollary}

According to Corollary \ref{col_nonconcave}, \FWPP finds an $\epsilon$-FOSP for general stochastic non-concave problems after at most $\mathcal{O}(1/\epsilon^3)$ stochastic gradient evaluations.

\begin{remark}
The results in Theorem \ref{thm_nonconvex} and Corollary \ref{col_nonconcave} hold for the general non-oblivious problem~\eqref{eqn_DR_submodular_maximization_loss}. Indeed, such complexity bounds also hold for the oblivious setting and the proof follows similarly (with requiring less assumptions). More precisely, to prove the same theoretical guarantees for the oblivious case, we only require Assumptions \ref{ass_compact_domain}, \ref{ass_smooth} and the boundedness of variance $\EBB_{\zB\sim p(\zB)}\|\tilde{F}(\xB;\zB) - F(\xB)\|^2\leq G^2$.
\label{remark_convex}
\end{remark}

\subsection{Convergence Analysis of \FWPP: Convex Setting}
In this section, we establish the complexity of \FWPP for finding an $\epsilon$-approximate solution when the function $F$ in \eqref{eqn_convex_minimization_loss} is convex or equivalently the function $F$ in \eqref{eqn_DR_submodular_maximization_loss} is concave.

\begin{theorem}	\label{thm_convex}
Consider Problem \eqref{eqn_DR_submodular_maximization_loss} when $F$ is a concave function. Further, recall the \FWPP method outlined in Algorithm \ref{algo_FW++}. Suppose the conditions in Assumptions \ref{ass_upper_bound_stochastic_function_value}, \ref{ass_compact_domain}, \ref{ass_variance}, \ref{ass_smooth}, and \ref{ass_Hessian_continuous} are satisfied. Further, let $\bar{L}^2 \defi 4B^2G^4 + 16G^4  + 4L^2 + 4B^2L^2$. If we set \FWPP parameters to $\eta_t = {2}/({t+2})$,  $|\MM_h^t| \!=\! 16(t\!+\!2)$ and $|\MM_0^{t}| = ({G^2(t\!+\!1)^2})/({\bar{L}^2D^2})$, then the  iterates generated by \FWPP satisfy
\vspace{-1.5mm}
	\begin{equation*}
		F(\xB^*) - \EBB [F(\xB^t)]\leq \frac{28\bar{L}D^2 + (F(\xB^*) - F(\xB^0))}{t+2}.\vspace{-1.5mm}
	\end{equation*}
\end{theorem}

Theorem \ref{thm_convex} shows that after at most $\mathcal{O}(1/\epsilon)$ iterations  and $\mathcal{O}(1/\epsilon)$ calls to a linear minimization oracle \FWPP reaches an $\epsilon$-approximate solution. Next we characterize the overall complexity of \FWPP in terms of stochastic gradient evaluations.
\vspace{-.1cm}
\begin{corollary}	\label{col_concave}
	Assume that $\epsilon$ satisfies $t\! =\! ({28LD^2\! +\! (F(\xB^*) \!-\! F(\xB^0))})/{\epsilon}  \!=\! 2^K$ for some $K \in \NBB$. Then, the overall stochastic complexity is 
	\vspace{-1.5mm}
	\begin{equation}
		\sum_{i=0}^t |\MM_h^i| + \sum_{k=0}^{K} |\MM_0^{2^k}| \!=\! \OM\left(\frac{\bar{L}^2D^4}{\epsilon^2} \!+\! \frac{G^2D^2}{\epsilon^2} \!+\! \frac{G^2(F(\xB^*) \!- \!F(\xB^0))^2}{\bar{L}^2D^2\epsilon^2}\right).\vspace{-1.5mm}
	\end{equation}
\end{corollary}

According to Corollary \ref{col_concave}, \FWPP finds an $\epsilon$-approximate solution for stochastic concave maximization (equivalently convex minimization) after at most computing $\mathcal{O}(1/\epsilon^2)$ stochastic gradient evaluations. 

\begin{remark}
The results in Theorem \ref{thm_convex} and Corollary \ref{col_concave} hold for the general non-oblivious problem in~\eqref{eqn_DR_submodular_maximization_loss} when the objective function $F$ is concave. Indeed, such complexity bounds also hold for the oblivious setting and the proof follows similarly (with requiring less assumptions). 
More precisely, the same theoretical guarantees as in Remark \ref{remark_convex} holds for the oblivious case under the Assumptions \ref{ass_compact_domain} and \ref{ass_smooth} and the bounded variance assumption $\EBB_{\zB\sim p(\zB)}\|\tilde{F}(\xB;\zB) - F(\xB)\|^2\leq G^2$.
\end{remark}

\section{Stochastic Continuous DR-submodular Maximization}
In this section, we focus on a special case of the non-oblivious problem in \eqref{eqn_DR_submodular_maximization_loss} when the function $F$ is continuous DR-submodular. We study both monotone and non-monotone settings and for each of them we present a new stochastic variant of the continuous greedy method \cite{calinescu2011maximizing} that can be interpreted as a conditional gradient method. We then extend our results to the problem of maximizing discrete submodular set functions when the objective is defined as an expectation of a collection of random set functions.

\subsection{Stochastic Continuous $\text{Greedy}{++}$: Monotone Setting} \label{section_scgpp}
We present $\scg$ (SCG++) which is the first method to obtain a $[(1-1/e)\OPT-\epsilon]$ solution with $O(1/\epsilon^2)$ stochastic oracle complexity for maximizing monotone but stochastic DR-submodular functions over a compact convex body.
 \SCGPP essentially operates in a conditional gradient manner. To be more precise, at each iteration $t$, given a gradient estimator $\gB^{t}$, \SCGPP solves the subproblem
\vspace{-1mm}
\begin{equation}\label{direction_update}
	\vB^t = \argmax_{\vB \in \CM} \langle \vB, \gB^{t}\rangle\vspace{-1mm}
\end{equation}
to obtain $\vB^t$ in $\CM$ as ascent direction, which is then added to the iterate $\xB^{t+1}$ with a scaling factor ${1}/{T}$, i.e., the new iterate $\xB^{t+1}$ is computed by following the update
\vspace{-1mm}
\begin{equation}\label{var_update}
\xB^{t+1} = \xB^{t} +\frac{1}{T}\vB^t,\vspace{-1mm}
\end{equation}
where $T$ is the total number of iterations of the algorithm.
Note the difference between \eqref{var_update} and \eqref{eqn_update_fw}.
The iterates are assumed to be initialized at the origin which may not belong to the feasible set $\CM$. Though each iterate $\xB^{t}$ may not necessarily be in $\CM$, the feasibility of the final iterate $\xB^{T}$ is guaranteed by the convexity of $\CM$.
Note that the iterate sequence $\{\xB^{s}\}_{s=0}^T$ can be regarded as a path from the origin (as we manually force $\xB^0 = 0$) to some feasible point in $\CM$.
The key idea in \SCGPP is to exploit the high correlation between the consecutive iterates originated from the $\OM({1}/{T})$-sized increments to maintain a highly accurate estimate $\gB^{t}$, which is evaluated based on the gradient estimation scheme presented in Section \ref{sec:sfw}. Note that by replacing the gradient approximation vector $\gB^t$ in the update of \SCGPP by the exact gradient of the objective function, we recover the update of continuous greedy \cite{calinescu2011maximizing, bian2017guaranteed}.

\begin{algorithm}[t]
	\caption{\texttt{Stochastic (Measured) Continuous Greedy}{++}}
	\label{algo_scg++}
	\begin{algorithmic}[1]
		\REQUIRE Minibatch size $|\MM_0|$ and $|\MM|$, and total number of rounds $T$
		\STATE Initialize $\xB^0 = 0$;
		\FOR{$t = 1$ \TO $T$}
		\IF{$t = 1$}
		\STATE Sample a minibatch $\MM_0$ of $\zB$ based on $p(\zB;\xB^0)$ and find $\gB^0 \defi \nabla \tilde{F}(\xB^0; \MM_0) $;
		\ELSE
		\STATE 
		Sample a minibatch $\MM$ of $\zB$ according {{ to $p(\zB;\xB(a))$ where $a$ is a chosen uniformly at random from $[0,1]$ and $\xB(a) := a\cdot\xB^{t}+(1-a)\cdot \xB^{t-1}$}};
		\STATE Compute the Hessian approximation $\tilde{\nabla}_{t}^2$ corresponding to $\MM$ based on \eqref{eqn_integral_estimator};
		\STATE Construct $\tilde{\Delta}^t$ based on \eqref{eqn_gradient_difference_estimator} (Option I) or \eqref{eqn_approximate_matrix_vector_product} (Option II);
		\STATE Update the stochastic gradient approximation $\gB^t := \gB^{t-1} + \tilde{\Delta}^t;$
		\ENDIF
		\STATE Set feasible set $\CM^t = \CM$ (\SCGPP) or $\CM^t = \{\vB\in\CM| \vB\leq \bar \uB - \xB^{t}\}$ (\NMSCGPP);
		\STATE Compute the ascent direction $\vB^t := \argmax_{\vB\in\CM^{t}}\{\vB^\top\gB^t\}$;
		\STATE Update the variable $\xB^{t+1} := \xB^t + 1/{T}\cdot{\vB^t}$;
		\ENDFOR
	\end{algorithmic}
\end{algorithm}
We proceed to analyze the convergence property of Algorithm \ref{algo_scg++} using \eqref{eqn_approximate_matrix_vector_product} as the gradient-difference estimation. 
Similar results can be obtained by using \eqref{eqn_gradient_difference_estimator}. 
We first specify the extra assumptions required for the analysis of the \SCGPP method.

\begin{assumption}\label{ass_upper_bound_function_value}
	The function $F$ satisfies $F(\mathbf{0})\geq 0$. 
\end{assumption}

\begin{assumption}
	$F$ is DR-submodular.
	\label{ass_function_submodularity}
\end{assumption}

\begin{assumption}
	$F$ is monotone.
	\label{ass_function_monotone}
\end{assumption}

Next, we incorporate the bound on the noise of gradient approximation presented in Lemma~\ref{lemma_variance_general} to characterize the convergence guarantee of \SCGPP.
\lightblue{Note that the following result appears as Theorem 1 in \cite{NIPS2019_9466}}.

\begin{theorem}
Consider \SCGPP outlined in Algorithm \ref{algo_scg++} and assume that in Step 8 we use the update in \eqref{eqn_approximate_matrix_vector_product} to find the gradient difference approximation $\tilde{\Delta}^t$.
	If Assumptions \ref{ass_upper_bound_stochastic_function_value}-\ref{ass_function_monotone} hold, then the output of  \SCGPP denoted by $\xB^T$ satisfies
	\vspace{-1.5mm}
	\begin{equation*}
	\EBB \left[F(\xB^{T})\right]\geq (1-1/e)F(\xB^*) - 2L\epsilon D^2,\vspace{-1.5mm}
	\end{equation*}
	by setting $|\MM_0| = \frac{G^2}{2\bar{L}^2D^2\epsilon^2}$, $|\MM| = \frac{1}{2\epsilon}$, $T = \frac{1}{\epsilon}$, and $\delta = \OM(\epsilon^2)$.
	Here $\bar{L}$ is a constant defined as $\bar{L}^2 \defi 4B^2G^4 + 16G^4  + 4L^2 + 4B^2L^2$.
	\label{thm_main}
\end{theorem}

The result in Theorem \ref{thm_main} shows that after at most $T=1/\epsilon$ iterations the objective function value for the output of \SCGPP is at least $(1-1/e)\OPT -\mathcal{O}(\epsilon)$. As the number of calls to the stochastic oracle per iteration is of $\mathcal{O}(1/\epsilon)$, to reach a $(1-1/e)\OPT -\mathcal{O}(\epsilon)$ approximation guarantee the \SCGPP  method has an overall stochastic first-order oracle complexity of $\mathcal{O}(1/\epsilon^2)$. We formally characterize this result in the following corollary. 

\vspace{-1mm}
\begin{corollary}
To find a $[(1-1/e)\OPT - {\epsilon}]$ solution to Problem (\ref{eqn_DR_submodular_maximization_loss}) using Algorithm \ref{algo_scg++} with Option II, the overall stochastic first-order oracle complexity is $({2G^2D^2 \!+\! 4\bar{L}^2D^4})/{{\epsilon}^2}$ and the overall linear optimization oracle complexity is ${2\bar{L}D^2}/{{\epsilon}}$.
	\label{corollary_complexity_submodular}
\end{corollary}
\vspace{-5mm}

\subsection{Stochastic Continuous $\text{Greedy}{++}$: Non-monotone Setting}
In this section, we consider the maximization of a non-monotone stochastic DR-submodular function. To present our method for solving this class of problems we first need to specify the domain $\mathcal{X}$ of the expected function $F:\mathcal{X}\to\mathbb{R}_{+}$ which is given by $\mathcal{X}=\prod_{i=1}^d \mathcal{X}_i $ where each $\mathcal{X}_i=[\underbar{$u$}_i,\bar{u}_i]$ is a bounded interval. To simplify the notation we define the vectors $\underbar{$\uB$}=[\underbar{$u$}_1,\dots,\underbar{$u$}_d]$ and $\bar{\uB}=[\bar{u}_1,\dots, \bar{u}_d]$. In this section, we further assume that the constraint set $\mathcal{C}$ is down-closed, i.e., $\mathbf{0}\in \mathcal{C}$.
It is known that without this assumption, no constant factor approximation guarantee is possible \cite{vondrak2008optimal}.

We present \nmscg (\NMSCGPP) for solving stochastic non-monotone DR-submodular functions in Algorithm~\ref{algo_scg++}. Note that many of the steps of \NMSCGPP are similar to the ones for \SCGPP, except feasible set selection in Step 11 which leads to a different update of ascent direction $\vB^t$. In particular, we compute the ascent direction in \NMSCGPP by solving the problem 
\vspace{-1.5mm}
\begin{equation}
\vB^t := \argmax_{\{\vB\in\CM| \vB\leq \bar \uB - \xB^{t}\}}\{\vB^\top\gB^t\},\vspace{-1.5mm}
\end{equation}
where $\gB^t$ is the stochastic gradient approximation at step $t$. This update differs from the update in \eqref{direction_update} for the monotone setting by having the extra constraint $\vB\leq \bar \uB - \xB^{t}$. This extra constraint is required to ensure that the outcome of this linear optimization does not grow aggressively as suggested previously in \cite{feldman2011unified,DBLP:conf/nips/BianL0B17}. In the following theorem, we show that \NMSCGPP obtains a $1/e$-guarantee.

\begin{theorem}
	Consider the \NMSCGPP  method outlined in Algorithm \ref{algo_scg++} and assume that in Step 8 we follow the update in \eqref{eqn_approximate_matrix_vector_product} to construct the gradient difference approximation $\tilde{\Delta}^t$ (Option II).
	If Assumptions \ref{ass_upper_bound_stochastic_function_value}-\ref{ass_function_submodularity}  hold, then the output of  \NMSCGPP  denoted by $\xB^T$ satisfies
	\begin{equation*}
	\EBB \left[F(\xB^{T})\right]\geq (1/e)F(\xB^*) - (4\sqrt{2}+1)/2\cdot\bar LD^2\epsilon,
	\end{equation*}
	by setting $|\MM_0| = \frac{G^2}{2\bar{L}^2D^2\epsilon^2}$, $|\MM| = \frac{1}{2\epsilon}$, $T = \frac{1}{\epsilon}$, and $\delta = \OM(\epsilon^2)$ as in (\ref{eqn_small_delta}).
	Here $\bar{L}$ is a constant defined by $\bar{L}^2 \defi 4B^2G^4 + 16G^4  + 4L^2 + 4B^2L^2$.
	\label{thm_main_nm}
\end{theorem}

The result in Theorem \ref{thm_main_nm} implies that after at most $T=1/\epsilon$ iterations the objective function value for the output of \NMSCGPP  is at least $(1/e)\OPT -\mathcal{O}(\epsilon)$. As the number of calls to the stochastic oracle per iteration is of $\mathcal{O}(1/\epsilon)$, to reach a $(1/e)\OPT -\mathcal{O}(\epsilon)$ approximation guarantee \NMSCGPP has an overall stochastic first-order oracle complexity of $\mathcal{O}(1/\epsilon^2)$ with $\mathcal{O}(1/\epsilon)$ calls to a linear optimization oracle.


\section{Stochastic Discrete Submodular Maximization} \label{discrete}

In this section, we focus on extending our result to the case where $F$ is the multilinear extension of a discrete submodular function $f$. This is also an instance of the non-oblivious stochastic optimization \eqref{eqn_DR_submodular_maximization_loss}.
Indeed, once such a result is achieved, with a proper rounding scheme such as pipage rounding \cite{calinescu2007maximizing} or contention resolution method \cite{chekuri2014submodular}, we can obtain discrete solutions. 
Let $V$ denote a finite set of $d$ elements, i.e., $V = \{1, \ldots, d\}$.
Consider a discrete submodular function $f: 2^V \rightarrow \RBB_+$, which is defined as an \emph{expectation} over a set of functions $f_\gamma:2^{V} \rightarrow \RBB_+$. Our goal is to maximize $f$ subject to some constraint $\mathcal{I}$, where the $\mathcal{I}$ is a collection of the subsets of $V$, i.e., we aim to solve the following discrete and stochastic submodular function maximization problem
\vspace{-1mm}
\begin{equation} \label{eqn_stochastic_submodular_function_definition}
\max_{S \in \mathcal{I}} f(S) := \max_{S \in \mathcal{I}} \EBB_{\gamma \sim p(\gamma)}[f_\gamma(S)],
\vspace{-1mm}
\end{equation}
where $p(\gamma)$ is an arbitrary distribution. 
In particular, we assume the pair $M = \{V, \mathcal{I}\}$ forms a matroid with rank $r$.
The prototypical example is maximization under the cardinality constraint, i.e., for a given integer $r$, find $S \subseteq V$, $|S|\leq r$, which maximizes $f$. The challenge here is to find a  solution with near-optimal quality for the problem in \eqref{eqn_stochastic_submodular_function_definition} without computing the expectation in \eqref{eqn_stochastic_submodular_function_definition}. That is, we assume access to an oracle that, given a set $S$, outputs an independently chosen sample $f_\gamma(S)$ where $\gamma\sim p(\gamma)$.   The focus of this section is on extending our results into the discrete domain and showing that \SCGPP can be applied for maximizing a stochastic submodular set function $f$, namely Problem \eqref{eqn_stochastic_submodular_function_definition}, through the multilinear extension of the function $f$.
Specifically, in lieu of solving \eqref{eqn_stochastic_submodular_function_definition} we can solve its continuous extension 
 \vspace{-1mm}
\begin{equation} \label{eqn_multilinear_extension_problem}
 \max_{\xB\in \mathcal{C}}\ F(\xB),\vspace{-1mm}
\end{equation} 
 where  $F: [0,1]^V \to \mathbb{R}_+$ is the multilinear extension of $f$ and is defined as 
\begin{equation} \label{multilinear} 
	F(\xB) \ :=\ \sum_{S \subseteq V } f(S) \prod_{i \in S} x_i \prod_{j \notin S} (1-x_j) \ =\ \sum_{S \subseteq V } \EBB_{\gamma \sim p(\gamma)}[f_\gamma(S)] \prod_{i \in S} x_i \prod_{j \notin S} (1-x_j),
\end{equation} 
and the convex set $\mathcal{C}=\text{conv}\{1_{I}: I\in \mathcal{I} \}$ is the matroid polytope \cite{calinescu2007maximizing}. Note that here $x_i $ denotes the $i$-th component of $\xB$. In other words, $F(\xB)$ is the expected value of $f$ over sets wherein each element $i$ is included with probability $x_i$ independently.

To solve the multilinear extension problem in \eqref{multilinear} using \SCGPP (for the monotone case) and \NMSCGPP (for the non-monotone case), we need access to unbiased estimators of the gradient and the Hessian. In the following lemma, we first study the structure of the Hessian of the objective function \eqref{multilinear}.

\vspace{-1mm}

\begin{lemma}[\cite{calinescu2007maximizing}]
Recall the definition of $F$ in \eqref{multilinear} as the multilinear extension of the set function $f$ in \eqref{eqn_stochastic_submodular_function_definition}. Then, for $i =j$ we have $[\nabla^2 F(\yB)]_{i,j} = 0$, and for $i\neq j$ 
\vspace{-1mm}
		\begin{align}	\label{eqn_multilinear_extension_lemma}
		[\nabla^2 F(\yB)]_{i,j}= &F(\yB; \yB_i\leftarrow 1, \yB_j \leftarrow 1) - F(\yB; \yB_i\leftarrow 1, \yB_j \leftarrow 0) \nonumber\\
		&- F(\yB; \yB_i\leftarrow 0, \yB_j \leftarrow 1) + F(\yB; \yB_i\leftarrow 0, \yB_j \leftarrow 0),\vspace{-1mm}
		\end{align}
	where the vector $\yB;\yB_i\leftarrow c_i, \yB_j\leftarrow c_j$ is defined as a vector that the $i^{th}$ and $j^{th}$ entries of $\yB$ is set to $c_i$ and $c_j$, respectively. 
	\label{lemma_multilinear_extension_lemma}
\end{lemma}

\vspace{-1mm}

Note that each term in \eqref{eqn_multilinear_extension_lemma} is an expectation which can be estimated in a bias-free manner by direct sampling.
We will now construct the Hessian approximation $\tilde{\nabla}_{t}^2 $ using Lemma \ref{lemma_multilinear_extension_lemma}.
Let $a$ be a uniform random variable between $[0, 1]$ and let $\eB = (e_1, \cdots, e_d)$ be a random vector in which $e_i$'s are generated i.i.d. according to the uniform distribution over the unit interval $[0,1]$.
In each iteration, a minibatch $\MM$ of $|\MM|$ samples of $\{a, \eB, \gamma\}$ (recall that $\gamma$ is the random variable that parameterizes the  component function $f_\gamma$), i.e., $\MM = \{a_k, \eB_k, \gamma_k\}_{k=1}^{|\MM|}$, is generated.
Then for all $k \in [|\MM|]$, we let $\xB_{a_k} = a_k\xB^{t} + (1-a_k)\xB^{t-1}$ and construct the random set $S(\xB_{a_k},\eB_k)$ using $\xB_{a_k}$ and $\eB_k$ in the following way: $s \in S(\xB_{a_k},\eB_k)$ if and only if $[\eB_k]_s \leq [\xB_{a_k}]_s$ for $s\in[d]$.
Having $S(\xB_{a_k},\eB_k)$ and $\gamma_k$, each entry of the Hessian estimator $\tilde{\nabla}_{t}^2 \in \RBB^{d\times d}$ is 
\vspace{-2mm}
\begin{equation}
\begin{aligned}
\left[\tilde{\nabla}_{t}^2 \right]_{i,j} = \frac{1}{|\MM|}\sum_{k\in[|\MM|]}f_{\gamma_k}(S(\xB_{a_k},\eB_k) \cup \{i, j\}) - f_{\gamma_k}(S(\xB_{a_k},\eB_k) \cup \{i\}\setminus \{j\}) \\
- f_{\gamma_k}(S(\xB_{a_k},\eB_k) \cup \{j\}\setminus \{i\}) + f_{\gamma_k}(S(\xB_{a_k},\eB_k) \setminus \{i,j\}),
\end{aligned}\vspace{-1mm}
\label{eqn_multilinear_extension_hessian_stochastic_ineqj}
\end{equation}
where $i\neq j$, and if $i = j$ then $[\tilde{\nabla}_{t}^2 ]_{i,j} = 0$.
As linear optimization over the rank-$r$ matroid polytope returns $\vB^t$ with at most $r$ nonzero entries, the complexity of computing \eqref{eqn_multilinear_extension_hessian_stochastic_ineqj} is $\OM(rd)$.
Now we use the above approximation of Hessian  to solve the multilinear extension as a special case of Problem \eqref{eqn_DR_submodular_maximization_loss} using \SCGPP and \NMSCGPP. To do so, we first introduce the following definitions.
\vspace{-1mm}
\begin{definition} \label{ass_bounded_marginal_return}
	Let $D_\gamma$ denote the maximum marginal value of $f_\gamma$, i.e., 
	$D_\gamma = \max_{i \in V} f_\gamma({i})$, and further define $D_f = ({\EBB_{\gamma} [D_\gamma^2]})^{1/2}$.
\end{definition}
\vspace{-1mm}

Based on Definition \ref{ass_bounded_marginal_return}, the Hessian estimator $\tilde{\nabla}_{t}^2$ has a bounded $\|\cdot\|_{2, \infty}$ norm:
$\EBB [\|\tilde{\nabla}_{t}^2\|_{2,\infty}^2] = \EBB [\max_{i \in [d]} \|\tilde{\nabla}_{t}^2(:, i)\|^2]\leq 4d\cdot\EBB_\gamma D_\gamma^2 = 4d\cdot D_f^2$.

\subsection{Convergence results}
We first analyze the convergence of \SCGPP for solving Problem \eqref{eqn_multilinear_extension_problem} when $f$ is monotone. Compared to Theorem \ref{thm_main}, Theorem \ref{thm:multi} has a dependency on the problem dimension $d$ and exploits the sparsity of $\vB^t$.
\lightblue{Note that this result is presented as Theorem 2 in \cite{NIPS2019_9466}.}

\begin{theorem}\label{thm:multi}
	Consider the multilinear extension of a monotone stochastic submodular set function and recall the definition of $D_f$.
	By using the minibatch size $|\MM| = \OM({\sqrt{r^3d}D_f}/{\epsilon})$ and $|\MM_0| = \OM({\sqrt{d}D_f}/{\sqrt{r}\epsilon^2})$, \SCGPP finds a $[(1-1/e)OPT - 6\epsilon]$ approximation of the multilinear extension problem at most $({\sqrt{r^3d}D_f}/{\epsilon})$ iterations. Moreover, the overall stochastic oracle cost is $\OM({r^3 d D_f^2}/{\epsilon^2})$.
\end{theorem}

Since the cost of a single stochastic gradient computation is $\OM(d)$, Theorem \ref{thm:multi} shows that the overall computation complexity of Algorithm \ref{algo_scg++} is $\OM({d^2}/{\epsilon^2})$.
Note that, in the multilinear extension case, the {{smoothness Assumption \ref{ass_smooth}}} required for the results in Section~\ref{section_scgpp} is absent, and that is why we need to develop a more sophisticated gradient-difference estimator to achieve a similar theoretical guarantee.


%

\begin{remark}[optimality of oracle complexities]
	In order to achieve the tight  $1-1/e-\epsilon$ approximation, the stochastic oracle complexity $\OM({1}/{\epsilon^2})$, obtained in Theorem  \ref{thm:multi}, is optimal in terms of its dependency on $\epsilon$. A lower bound on the stochastic oracle complexity is given in the following section.
\end{remark}

We proceed to derive our result for stochastic and non-monotone discrete submodular function maximization. 

\begin{theorem}\label{thm:nm-multi}
	Consider the multilinear extension of a non-monotone stochastic submodular set function and recall the definition of $D_f$.
	By using the minibatch size $|\MM| = \OM({\sqrt{r^3d}D_f}/{\epsilon})$ and $|\MM_0| = \OM({\sqrt{d}D_f}/{\sqrt{r}\epsilon^2})$, \NMSCGPP finds a $[(1/e)OPT - \epsilon]$ approximation of the multilinear extension problem at most $({\sqrt{r^3d}D_f}/{\epsilon})$ iterations. Moreover, the overall stochastic oracle cost is $\OM({r^3 d D_f^2}/{\epsilon^2})$.
\end{theorem}



\subsection{Lower Bound}

In this section, we show that reaching a $(1-1/e - \epsilon)$-optimal solution of Problem~\eqref{eqn_DR_submodular_maximization_loss} when $F$ is a monotone DR-submodular function requires at least $\mathcal{O}(1/\epsilon^2)$ calls to an oracle that provides stochastic first-order information. To do so, we first construct a stochastic submodular set function $f$, defined as $f(S) = \EBB_{\gamma \sim p(\gamma)}[f_\gamma(S)]$, with the following property: Obtaining a $(1-1/e - \epsilon)$-optimal solution for maximization of $f$ under a cardinality constraint requires at least $\mathcal{O}(1/\epsilon^2)$ samples of the form $f_\gamma(\cdot)$ where $\gamma$ is generated i.i.d from distribution $p$.  Such a lower bound on sample complexity can be directly extended to Problem~\eqref{eqn_DR_submodular_maximization_loss} with a stochastic first order oracle, by considering the multilinear extension of the function $f$, denoted by $F$, 
and noting that (i) problems~\eqref{eqn_stochastic_submodular_function_definition} and \eqref{eqn_multilinear_extension_problem} have the same optimal values, and (ii) one can construct an unbiased estimator of the gradient of the multilinear extension using $d$ independent samples from the underlying stochastic set function $f$.  Hence, any method for maximizing \eqref{eqn_multilinear_extension_problem} is also a method for maximizing \eqref{eqn_stochastic_submodular_function_definition} with the same guarantees on the quality of the solution and with sample complexities that differ at most by a factor of $d$. Now formalize the above argument. \lightblue{We note that this result is presented as Theorem 3 in \cite{NIPS2019_9466}}.

\begin{theorem} \label{sub_lower}
There exists a distribution $p(\gamma)$ and a monotone submodular function $f: 2^V \to \mathbb{R}_+$, given as $ f(S) = \EBB_{\gamma \sim p(\gamma)}[f_\gamma(S)]$, such that the following holds:  In order to find a $(1-1/e - \epsilon)$-optimal solution for \eqref{eqn_stochastic_submodular_function_definition} with $k$-cardinality constraint,  any algorithm requires at least $\min\{\exp(\alpha k), {\beta}/{\epsilon^2}\}$ stochastic samples $f_\gamma(\cdot)$. 
\end{theorem}

\begin{corollary}
There exists a monotone DR-submodular function $F$, a convex constraint $\mathcal{C}$, and a stochastic first-order oracle $\mathcal{O}_{first}$, such that any method for maximizing $F$ subject to $\mathcal{C}$ requires at least $\min\{\exp(\alpha n), \beta/\epsilon^2\}$ queries from $\mathcal{O}_{first}$. 
\end{corollary}




\vspace{-2mm}
\section{Conclusion}

In this paper we studied a class of stochastic conditional gradient methods for solving non-oblivious convex and nonconvex minimization problems as well as continuous DR-submodular maximization problems. In particular, (i) we proposed a stochastic variant of the Frank-Wolfe method called \FWPP for
minimizing a  smooth \emph{non-convex} stochastic function  subject to a convex body constraint. We showed that \FWPP finds an $\epsilon$-first order stationary point after at most $O(1/\epsilon^3)$ stochastic gradient evaluations; (ii) we further studied the convergence rate of \FWPP when we face a constrained \emph{convex} minimization problem and showed that \FWPP achieves an $\epsilon$-approximate optimum while using $O(1/\epsilon^2)$ stochastic gradients; (iii) we also extended the idea of our proposed variance reduced stochastic condition gradient method to the \emph{submodular} setting and developed \SCGPP, the first efficient variant of continuous greedy for maximizing a stochastic, continuous, monotone  DR-submodular function subject to a convex constraint. We showed that \SCGPP achieves a tight $[(1-1/e)\OPT -\epsilon]$ solution while using $O(1/\epsilon^2)$ stochastic gradients. We further derived a tight lower bound on the number of calls to the first-order stochastic oracle for achieving a $[(1-1/e)\OPT -\epsilon]$ approximate solution. This result showed that \SCGPP has the optimal sample complexity for finding an optimal $(1-1/e)$ approximation guarantee for monotone but stochastic DR-submodular functions. Finally, for maximizing a non-monotone continuous DR-submodular function, \SCGPP achieves a  $[(1/e)\OPT -\epsilon]$ solution after computing $O(1/\epsilon^2)$ stochastic gradients.

\appendix

%
%

\section{Proof of Lemma \ref{lemma_variance_general}}
We first present a lemma which bounds the second moment of the spectral norm of the Hessian estimator $\nabla^2 \tilde{F}(\yB; \zB)$.
\begin{lemma}
	Recall the definition of the Hessian estimator $\tilde\nabla^2 {F}(\yB; \zB)$ in \eqref{eqn_unbiased_second_order_differential_estimator}.
	Under Assumptions \ref{ass_upper_bound_stochastic_function_value}, \ref{ass_variance}, \ref{ass_smooth}, we bound for any $\yB \in \CM$
\begingroup
\setlength\abovedisplayskip{5pt}
	\begin{equation}
		\EBB_{\zB\sim p(\zB; \yB)}[\|\tilde\nabla^2 {F}(\yB; \zB)\|^2] \leq 4B^2G^4 + 16G^4  + 4L^2 + 4B^2L^2 \defi \bar{L}^2.\vspace{-3mm}
	\end{equation}
	\endgroup
	\label{lemma_appendix_hessian_spectral_norm}
\end{lemma}
\vspace{-2mm}
\begin{proof}
	From the definition of $\tilde\nabla^2 {F}(\yB; \zB)$, we have from assumptions \ref{ass_upper_bound_stochastic_function_value} and \ref{ass_variance}
	\begin{equation}
		\|\tilde{\nabla}^2 F(\yB; \zB)\| \leq B \|\nabla\log p(\zB;\yB)\|^2 + 2G\|\nabla\log p(\zB;\yB)\|  + L + B\|\nabla^2\log p(\zB;\yB)\|.
	\end{equation}
	Futher, taking expectation on both sides and use Assumption \ref{ass_smooth} to bound
	\begin{equation}
		\EBB[\|\tilde{\nabla}^2 F(\yB; \zB)\|^2] \leq 4B^2G^4 + 16G^4  + 4L^2 + 4B^2L^2.
	\end{equation}
	\end{proof}
\vspace{-2mm}
\begin{proof}[Lemma \ref{lemma_variance_general}]
	We prove via induction.
	When $t=0$, use the unbiasedness of $\nabla \tilde{F}(\xB^0; \zB)$ and Assumption \ref{ass_variance}, we bound
	\begingroup
	\setlength\abovedisplayskip{5pt}
	\begin{equation*}
	\begin{aligned}
		\EBB_{\MM_0} [\|F(\xB^{0}) - \gB^{0}\|^2] =& \frac{1}{|\MM_0|}\EBB[\|F(\xB^{0}) - \nabla \tilde{F}(\xB^0; \zB)\|^2] \\
		\leq& \frac{1}{|\MM_0|}\EBB[\|\nabla \tilde{F}(\xB^0; \zB)\|^2] \leq\frac{G^2}{|\MM_0|} \leq \bar{L}^2D^2\epsilon^2.
	\end{aligned}
	\end{equation*}
	\endgroup
	Now assume that we have the result for $t = \bar{t}$.
	When $t = \bar{t}+1$, we have
	\begin{equation*}
		\begin{aligned}
			\gB^t - \nabla F(\xB^t) =& \left[\gB^{t-1} - \nabla F(\xB^{t-1})\right] + \left[\xi_\delta(\xB;\MM) - \tilde{\nabla}_t^2(\xB^{t} - \xB^{t-1})\right] \\
			&+ \left[\tilde{\nabla}_t^2(\xB^{t} - \xB^{t-1}) - (\nabla F(\xB^{t}) - \nabla F(\xB^{t-1}))\right].
		\end{aligned}
	\end{equation*}
	Expand $\|\nabla F(\xB^t) - \gB^t\|^2$ to obtain
	\begin{align}
	\|\nabla F(\xB^t) \!-\! \gB^t\|^2 =& \|\nabla F(\xB^{t}) - \nabla F(\xB^{t-1}) - \tilde{\nabla}_t^2(\xB^{t} - \xB^{t-1})\|^2 + \|\gB^{t-1} - \nabla F(\xB^{t-1})\|^2 \notag\\
	& + 2\langle\nabla F(\xB^{t}) - \nabla F(\xB^{t-1}) - \tilde{\nabla}_t^2(\xB^{t} - \xB^{t-1}), \gB^{t-1} - \nabla F(\xB^{t-1})\rangle \notag\\
	& + 2\langle \tilde{\nabla}^2_t(\xB^{t} \!- \!\xB^{t-1}) - \xi_\delta(\xB;\MM), \nabla F(\xB^{t}) \!-\! \nabla F(\xB^{t-1}) \!-\! \tilde{\nabla}_t^2(\xB^{t}\! -\! \xB^{t-1})\rangle \notag\\
	& + 2\langle \tilde{\nabla}^2_t(\xB^{t} \!-\! \xB^{t-1}) - \xi_\delta(\xB;\MM), \gB^{t-1} - \nabla F(\xB^{t-1})\rangle \notag\\
	& + \|\tilde{\nabla}^2_t(\xB^{t} - \xB^{t-1}) - \xi_\delta(\xB;\MM)\|^2. \label{eqn_proof}
	\end{align}
	Using the unbiasedness of $\tilde{\nabla}_t^2(\xB^{t} - \xB^{t-1})$, we have
	\begin{equation}
	\EBB [\langle\nabla F(\xB^{t}) - \nabla F(\xB^{t-1}) - \tilde{\nabla}_t^2(\xB^{t} - \xB^{t-1}), \gB^{t-1} - \nabla F(\xB^{t-1})\rangle] = 0.
	\end{equation}
	Additionally, from the unbiasedness of $\tilde{\Delta}^t$, we have
	\begin{equation}
		\EBB [\|\tilde{\Delta}^t - (\nabla F(\xB^{t}) - \nabla F(\xB^{t-1}))\|^2]\leq\frac{\epsilon^2D^2}{|\MM|}\EBB[\|\nabla^2 F(\xB({a_1}); \zB_1({a_1}))\|^2] \leq \frac{\epsilon^2\bar{L}^2D^2}{|\MM|},
	\end{equation}
	where we use Lemma \ref{lemma_appendix_hessian_spectral_norm}.
	Taking expectation on both sides of (\ref{eqn_proof}), we have
	\begin{align*}
	&\EBB [\|\nabla F(\xB^t) - \gB^t\|^2] \\
	&=\EBB [\|\nabla F(\xB^{t}) - \nabla F(\xB^{t-1}) - \tilde{\nabla}_t^2(\xB^{t} - \xB^{t-1})\|^2] + \EBB [\|\gB^{t-1} - \nabla F(\xB^{t-1})\|^2] \notag\\
	&+ 2\EBB[\|\tilde{\nabla}^2_t(\xB^{t} - \xB^{t-1}) - \xi_\delta(\xB;\MM)\|\|\nabla F(\xB^{t}) - \nabla F(\xB^{t-1}) - \tilde{\nabla}_t^2(\xB^{t} - \xB^{t-1})\|] \notag\\
	&+\! 2\EBB[\|\tilde{\nabla}^2_t(\xB^{t} \!-\! \xB^{t-1}) \!-\! \xi_\delta(\xB;\MM)\|\|\gB^{t-1} \!-\! \nabla F(\xB^{t-1})\|] \!+\! \EBB[\|\tilde{\nabla}^2_t(\xB^{t} \!-\! \xB^{t-1}) \!-\! \xi_\delta(\xB;\MM)\|^2] \notag\\
	 &\leq \EBB [\|\nabla F(\xB^{t}) - \nabla F(\xB^{t-1}) - \tilde{\nabla}_t^2(\xB^{t} - \xB^{t-1})\|^2] + \EBB [\|\gB^{t-1} - \nabla F(\xB^{t-1})\|^2] \! +\! 4D^4L_2^2\delta^2\\
	&+ 4 D^2L_2\delta\|\nabla F(\xB^{t}) - \nabla F(\xB^{t-1}) - \tilde{\nabla}_t^2(\xB^{t} - \xB^{t-1})\| \!+\! 4 D^2L_2\delta\|\gB^{t-1} \!-\!\nabla F(\xB^{t-1})\| \\
	&\leq \frac{\bar{L}^2D^2\epsilon^2}{|\MM|} \!+\! (1\!+\!\epsilon{(t\!-\!1)})\bar{L}^2D^2\epsilon^2 \!+\! 4\delta\!\left[\frac{D^2L_2\bar{L}D\epsilon}{\sqrt{|\MM|}} \!+\!  D^2L_2\sqrt{(1\!+\!\epsilon{(t\!-\!1)})}\bar{L}D\epsilon \!+\! D^4L_2^2\delta\right]
	\end{align*}
	By taking $\delta$ sufficiently small such that 
	\begin{equation}
		4\delta\left((1/{\sqrt{|\MM|}}){D^2L_2\bar{L}D\epsilon} +  D^2L_2\sqrt{(1+\epsilon{(t-1)})}\bar{L}D\epsilon + D^4L_2^2\delta\right) \leq \bar{L}^2D^2\epsilon^3/2,
		\label{eqn_small_delta}
	\end{equation}
	we have shown that the induction holds for $t = \bar{t}+1$.
\end{proof}

\section{Proof of Theorem \ref{thm_nonconvex}}

First we prove the following lemma.

\begin{lemma}
	Recall the definition in \eqref{eqn_unbiased_second_order_differential_estimator}.
	Under Assumptions \ref{ass_upper_bound_stochastic_function_value}, \ref{ass_variance}, \ref{ass_smooth}, and
	by taking $q = G/(16\epsilon)$, $|\MM_0^t| = G^2/(8\epsilon^2)$, $|\MM_h^t| = 2G/\epsilon$, and $\eta_t = \epsilon/(\bar LD)$, we bound
	\begin{equation}
	\EBB [\|\gB^t - \nabla F(\xB^{t})\|^2] \leq \epsilon^2/4,
	\end{equation} 
	where $\bar{L}$ is defined in Lemma \ref{lemma_variance_general}.
	\label{lemma_nonconcave_var_bound}
\end{lemma}

\begin{proof}
	For $t$ such that $mod(t, p)\neq 0$,
	\begin{equation*}
	\begin{aligned}
	&\EBB [\|\gB^{t} - \nabla F(\xB^{t})\|^2]
	= \EBB [\|\tilde{\nabla}_t^2 [\xB^{t} - \xB^{t-1}] + \gB^{t-1} - \nabla F(\xB^{t})\|^2] \\
	=& \EBB [\|\tilde{\nabla}_t^2 [\xB^{t} - \xB^{t-1}] - (\nabla F(\xB^{t}) - \nabla F(\xB^{t-1}))\|^2] + \EBB [\|\gB^{t-1} - \nabla F(\xB^{t-1})\|^2] \\
	=& \frac{1}{|\MM_h^t|} \EBB [\|\tilde{\nabla}_{1}^2 [\xB^{t} - \xB^{t-1}]- (\nabla F(\xB^{t}) - \nabla F(\xB^{t-1}))\|^2] + \EBB [\|\gB^{t-1} - \nabla F(\xB^{t-1})\|^2]\\
	\leq& \frac{1}{|\MM_h^t|} \EBB [\|\tilde{\nabla}_{1}^2 [\xB^{t} - \xB^{t-1}]\|^2] + \EBB [\|\gB^{t-1} - \nabla F(\xB^{t-1})\|^2].
	\end{aligned}
	\end{equation*}
	Observe that $\xB^{t+1} - \xB^{t} = \eta_t (\vB^t -\xB^{t})$ and therefore
	\begin{equation*}
	\begin{aligned}
	\EBB [\|\gB^{t} - \nabla F(\xB^{t})\|^2] &\leq \frac{4\eta_t^2D^2}{|\MM_h^t|} \EBB [\|\tilde{\nabla}_{1}^2\|^2] + \EBB [\|\gB^{t-1} - \nabla F(\xB^{t-1})\|^2] \\
	&\leq \frac{4\eta_t^2D^2\bar{L}^2}{|\MM_h^t|} + \EBB [\|\gB^{t-1} - \nabla F(\xB^{t-1})\|^2],
	\end{aligned}
	\end{equation*}
	where we use Lemma \ref{lemma_appendix_hessian_spectral_norm} in the second inequality.
	Denote by $k_0 = q\times\lfloor t/q\rfloor$ and $k = mod(t, q) \leq q$.
	Repeat the above recursion $k$ times to obtain
	\begin{align*}
	\EBB [\|\gB^{t} - \nabla F(\xB^{t})\|^2] &\leq \frac{4k\eta_t^2D^2\bar{L}^2}{|\MM_h^t|} + \EBB [\|\gB^{k_0} - \nabla F(\xB^{k_0})\|^2]= \frac{4q\eta_t^2D^2\bar{L}^2}{|\MM_h^t|} + \frac{G^2}{|\MM_0^t|}.
	\end{align*}
	By setting $q = G/(16\epsilon)$, $|\MM_0^t| = G^2/(8\epsilon^2)$, $|\MM_h^t| = 2G/\epsilon$, and $\eta_t = \epsilon/\bar LD$, we have
	\begin{equation}
	\EBB [\|\gB^t - \nabla F(\xB^{t})\|^2] \leq \epsilon^2/4.
	\end{equation}
\end{proof}

Now we are ready to prove the claim in Theorem \ref{thm_nonconvex}. 
From Lemma \ref{lemma_appendix_hessian_spectral_norm}, we have 
\begin{equation}
\|\nabla^2 F(\xB)\|^2 \leq \|\EBB_{\zB \sim p(\zB; \xB)} [\nabla ^2 \tilde{F}(\xB; \zB)]\|^2\leq \EBB_{\zB \sim p(\zB; \xB)}[\|\nabla ^2 \tilde{F}(\xB; \zB)\|^2]\leq \bar{L}^2.
\end{equation}
Hence, $F$ is $\bar{L}$-smooth.
From the smoothness of $F$
\begin{align*}
F(\xB^{t+1})                                                                                                         
\geq & ~  F (\xB^{t}) + \langle\nabla F (\xB^{t}), \xB^{t+1}-\xB^{t}\rangle - \frac{\bar L}{2}\|\xB^{t+1} - \xB^{t}\|^2             \\
=    & ~  F (\xB^{t}) + \langle \gB^{t}, \xB^{t+1} - \xB^{t}\rangle  - \frac{\bar L}{2}\|\xB^{t+1} - \xB^{t}\|^2   + \langle\nabla F (\xB^{t}) - \gB^{t}, \xB^{t+1} - \xB^{t}\rangle  \\
=    & ~  F (\xB^{t}) + \eta\langle \gB^{t}, \vB^{t}-  \xB^{t}\rangle  - \frac{\bar L\eta^2}{2}\|\uB^{t} - \xB^{t}\|^2 + \eta\langle\nabla F (\xB^{t}) - \gB^{t}, \vB^{t} - \xB^{t}\rangle                                           \\
\geq & ~  F (\xB^{t}) + \eta\langle \gB^{t}, \vB^{t}-  \xB^{t}\rangle - 2\bar L\eta^2D^2 - 2\eta D\|\nabla F (\xB^{t}) - \gB^{t}\|.
\end{align*}
Denoting by $\vB^+ = \argmax_{\vB\in C} \langle \nabla F(\xB^{t}), \vB -  \xB^{t}\rangle$, we have
\begingroup
\setlength\abovedisplayskip{5pt}
\begin{equation*}
\begin{aligned}
V_\CM(\xB^{t}; F) \!=\!  \langle \nabla F(\xB^{t}) \!-\! \gB^{t}, \vB^+ \!-\! \xB^{t}\rangle + \langle \gB^{t}, \vB^+ \!-\! \xB^{t}\rangle \leq  2D\|\nabla F(\xB^{t}) \!-\! \gB^{t}\| + \langle \gB^{t}, \vB^{t} \!-\! \xB^{t}\rangle.
\end{aligned}
\end{equation*}
\endgroup
These two bounds together give
\begingroup
\setlength\abovedisplayskip{5pt}
\begin{equation*}
\eta V_\CM(\xB^{t}; F) \leq F (\xB^{t+1}) - F (\xB^{t}) + 4\eta D\|\nabla F(\xB^{t}) - \gB^{t}\| + 2\bar L\eta^2D^2.
\end{equation*}
\endgroup
As $\eta={\epsilon}/{\bar LD}$ and $\EBB[\|\nabla F(\xB^{t}) - \gB^{t}\|^2] \leq {\epsilon^2}/{4}$ in Lemma \ref{lemma_nonconcave_var_bound}, we have for $t\geq 1$ 
\begin{align*}
\quad\frac{\epsilon}{\bar L D}\EBB[V_\CM(\xB^{t}; F)] \leq & \quad \EBB[F(\xB^{t+1})] - \EBB[F(\xB^{t})] + 2\epsilon^2/\bar L + 4\epsilon/\bar L\cdot\EBB[\|\nabla F(\xB^{t}) - \gB^{t}\|]                 \\
\leq & \quad \EBB[F(\xB^{t+1})] - \EBB[F(\xB^{t})] + {4\epsilon^2}/{\bar L}.
\end{align*}
Sum the above inequalities from $t=1$ to $T$ and multiply both sides by $\bar L D$ to obtain 
\vspace{-2mm}
\begin{equation*}
\sum_{t=1}^{T}{\epsilon}\EBB[V_\CM(\xB^{t}; f)]\leq \bar L D(F(\xB^*) - F(\xB^1)) + T\cdot{4D\epsilon^2}.\vspace{-2mm}
\end{equation*}
Hence, by sampling $t_0$ from $[T]$ uniformly at random, we have
\begin{equation*}
\EBB[V_\CM(\xB^{t}; F)] \leq \frac{D\bar L(F(\xB^*) - F(\xB^1))}{T\epsilon} + 4\epsilon D,
\end{equation*}
and thus when $T = {\bar L (F(\xB^*) - F(\xB^1))}/{\epsilon^2}$, we have $\EBB[V_\CM(\xB^{t}; F)] \leq 5\epsilon D$.

\section{Proof of Theorem \ref{thm_convex}}

We first prove the following lemma.

\begin{lemma}
	Recall the definition of the Hessian estimator in \eqref{eqn_unbiased_second_order_differential_estimator}.
	Under Assumptions \ref{ass_upper_bound_stochastic_function_value}, \ref{ass_variance}, \ref{ass_smooth},
	by taking $|\MM_h^t| = {16}{(t+2)}$ and $|\MM_0^{t}| = \frac{G^2(t+1)^2}{\bar{L}^2D^2}$, we bound
	\begin{equation}
	\EBB [\|\gB^t - \nabla F(\xB^{t})\|^2] \leq 2\bar{L}^2D^2\eta_{t}^2,
	\end{equation} 
	where $\bar{L}$ is defined in Lemma \ref{lemma_variance_general}.
	\label{lemma_concave_var_bound}
\end{lemma}

\begin{proof}
	Assume iteration $t$ is in the $k^{th}$ epoch, i.e., $2^k\leq t < 2^{k+1}$.
	For $t\neq 2^k$,
	\begin{equation*}
	\begin{aligned}
	&\EBB [\|\gB^{t} - \nabla F(\xB^{t})\|^2]
	= \EBB [\|\tilde{\nabla}_t^2 [\xB^{t} - \xB^{t-1}] + \gB^{t-1} - \nabla F(\xB^{t})\|^2] \\
	=& \EBB [\|\tilde{\nabla}_t^2 [\xB^{t} - \xB^{t-1}] - (\nabla F(\xB^{t}) - \nabla F(\xB^{t-1}))\|^2] + \EBB [\|\gB^{t-1} - \nabla F(\xB^{t-1})\|^2] \\
	=& \frac{1}{|\MM_h^t|} \EBB [\|\tilde{\nabla}_{1}^2 [\xB^{t} - \xB^{t-1}]- (\nabla F(\xB^{t}) - \nabla F(\xB^{t-1}))\|^2] + \EBB [\|\gB^{t-1} - \nabla F(\xB^{t-1})\|^2]\\
	\leq& \frac{1}{|\MM_h^t|} \EBB [\|\tilde{\nabla}_{1}^2 [\xB^{t} - \xB^{t-1}]\|^2] + \EBB [\|\gB^{t-1} - \nabla F(\xB^{t-1})\|^2].
	\end{aligned}
	\end{equation*}
	Observe that $\xB^{t+1} - \xB^{t} = \eta_t (\vB^t -\xB^{t})$ and therefore
	\begin{equation*}
	\begin{aligned}
	\EBB [\|\gB^{t} - \nabla F(\xB^{t})\|^2] &\leq \frac{4\eta_t^2D^2}{|\MM_h^t|} \EBB [\|\tilde{\nabla}_{1}^2\|^2] + \EBB [\|\gB^{t-1} - \nabla F(\xB^{t-1})\|^2] \\
	&\leq \frac{4\eta_t^2D^2\bar{L}^2}{|\MM_h^t|} + \EBB [\|\gB^{t-1} - \nabla F(\xB^{t-1})\|^2],
	\end{aligned}
	\end{equation*}
	where we use Lemma \ref{lemma_appendix_hessian_spectral_norm} in the second inequality.
	Repeat the above recursion $t - 2^k < {2^k}$ times (since $t< 2^{k+1}$), we obtain
	\begin{align*}
	&\EBB [\|\gB^{t} - \nabla F(\xB^{t})\|^2] \leq \EBB [\|\gB^{2^k} - \nabla F(\xB^{2^k})\|^2] + \sum_{i=2^k}^{t} 4D^2\bar{L}^2\cdot\frac{\eta_{i}^2}{|\MM_h^i|} \\
	&\leq \frac{G^2}{|\MM_0^{2^k}|} + \sum_{i=2^k}^{t} \frac{D^2\bar{L}^2}{(i+2)^3} 
	\leq \bar{L}^2D^2\eta_{2^{k+1}}^2 + \frac{D^2\bar{L}^2}{2}\sum_{i=2^k}^{t} \left[\frac{1}{(i+1)(i+2)} - \frac{1}{(i+2)(i+3)}\right] \\
	&\leq \bar{L}^2D^2\eta_{2^{k+1}}^2 + D^2\bar{L}^2 \frac{1}{(2^k+1)(2^{k-1}+1)} \leq 2\bar{L}^2D^2\eta_{2^{k+1}}^2 \leq 2\bar{L}^2D^2\eta_{t}^2
	\end{align*}
\end{proof}

Now we are ready to prove the claim in Theorem \ref{thm_convex}. 
From Lemma \ref{lemma_appendix_hessian_spectral_norm}, we have 
\begin{equation}
\|\nabla^2 F(\xB)\|^2 \leq \|\EBB_{\zB \sim p(\zB; \xB)} [\tilde\nabla ^2 {F}(\xB; \zB)]\|^2\leq \EBB_{\zB \sim p(\zB; \xB)}[\|\tilde\nabla ^2 {F}(\xB; \zB)\|^2]\leq \bar{L}^2.
\end{equation}
The boundedness of the Hessian $\nabla^2 F$ is equivalent to the smoothness of $F$.
Let $\xB^*$ be a global maximizer within the constraint set $\CM$. By the smoothness of $F$, we have
\begin{align}
F(\xB^{t+1}) &\geq F(\xB^t) + \langle\nabla F(\xB^t), \xB^{t+1} - \xB^t\rangle - ({\bar{L}}/{2})\|\xB^{t+1} - \xB^t\|^2 \notag\\
&= F(\xB^t) + \eta_{t}\langle\nabla F(\xB^t), \vB^t - \xB^{t}\rangle - ({\bar{L}\eta_{t}^2}/{2})\|\vB^t -\xB^{t}\|^2 \notag\\
&=F(\xB^t) + \eta_{t}\langle\gB^t, \vB^t -\xB^{t}\rangle + \eta_{t}\langle\nabla F(\xB^t) - \gB^t, \vB^t - \xB^{t}\rangle - 2\bar{L}\eta_{t}^2D^2 \notag \\
&\geq F(\xB^t) + \eta_{t}\langle\gB^t, \xB^* - \xB^{t}\rangle + \eta_{t}\langle\nabla F(\xB^t) - \gB^t, \vB^t - \xB^{t}\rangle - 2\bar{L}\eta_{t}^2D^2, 
\end{align}
where we use the optimality and boundedness of $\vB^t$ in the last inequality.
Take expectation, and use the unbiasedness of $\gB_{vr}^{t}$ and Young's inequality to obtain
\begin{align*}
\EBB [F(\xB^{t+1})] &\geq \EBB [F(\xB^t)] + \eta_{t}\langle\nabla F(\xB^{t}), \xB^* - \xB^{t}\rangle - \frac{1}{2\bar{L}}\EBB[\|\nabla F(\xB^t) - \gB^t\|^2] - 6\bar{L}\eta_{t}^2D^2.
\end{align*}
From the convexity of $F$, we have $\langle\nabla F(\xB^t), \xB^* -\xB^{t}\rangle \geq F(\xB^*) - F(\xB^t)$ and thus
\begin{equation}
\EBB [F(\xB^{t+1})]\geq \EBB [F(\xB^t)] + \eta_{t}\EBB[F(\xB^*) - F(\xB^t)] - \frac{1}{2\bar{L}}\EBB[\|\nabla F(\xB^t) - \gB^t\|^2] - 6\bar{L}\eta_{t}^2D^2.
\label{eqn_proof_convex_1}
\end{equation}
By using Lemma \ref{lemma_variance_general} with $|\MM_0| = \frac{G^2}{\bar{L}^2D^2\eta_{t}^2}$ and $|\MM_h| = \frac{1}{\eta_{t}}$, we have
\begin{equation}
\EBB[\|\nabla F(\xB^{t}) - \gB^{t}\|^2]\leq 2\bar{L}^2D^2\eta_{t}^2.
\label{eqn_proof_convex_2}
\end{equation}
Let $\delta_t \defi F(\xB^*) - F(\xB^{t})$ and $c \defi \max \{14\bar{L}D^2, \delta_0\}$.
Combining (\ref{eqn_proof_convex_1}) and (\ref{eqn_proof_convex_2}) gives
\begin{equation*}
\EBB [\delta_{t+1}] \leq (1-\eta_{t})\EBB [\delta_t ]+ c\eta_{t}^2/2.
\end{equation*}
By taking $\eta_t = \frac{2}{t+2}$ and by induction we obtain $\EBB \delta_{t} \leq \frac{2c}{t+2}$: For $t=0$, it trivially holds.
Assume $\EBB [\delta_{t_0}] \leq \frac{2c}{t_0+2}$ with $t_0\geq 1$. For $t = t_0+1$,
\begin{equation}
\EBB[ \delta_{t_0+1} ]\leq \frac{t_0}{t_0+2}\cdot\frac{2c}{t_0+2} + \frac{2c}{(t_0+2)^2} \leq \frac{2c}{t_0+3}.
\end{equation}
In conclusion, we have $
F(\xB^*) - \EBB [F(\xB^t)]\leq \frac{28\bar{L}D^2 + (F(\xB^*) - F(\xB^0))}{t+2}
$.

\section{Proof of Theorem \ref{thm_main}}
	Theorem \ref{thm_main} is identical to Theorem 1 of the conference version of this paper \cite{NIPS2019_9466}, and we refer the reader to the detailed proof therein. Here we only provide a sketch.
	From Lemma \ref{lemma_appendix_hessian_spectral_norm}, $F$ can be proved to be $\bar{L}$-smooth like in Theorem \ref{thm_nonconvex}.
	By using Lemma \ref{lemma_variance_general}, for all $t\in \{0, \ldots, T-1\}$ we have
		$\EBB[\|\nabla F(\xB^{t}) - \gB^{t}\|^2]\leq 2\bar{L}^2D^2\epsilon^2$.
	Let $\xB^*$ be the global maximizer within the constraint set $\CM$. 
	We can prove function value increases for $T = {1}/{\epsilon}$
	\vspace{-2mm}
	\begin{equation*}
	\EBB [F(\xB^{t+1})]\geq \EBB [F(\xB^t)] + \epsilon\EBB[F(\xB^*) - F(\xB^t)] - 2\bar{L}\epsilon^2D^2,
	\end{equation*}
	which is equivalent to $
	\EBB [F(\xB^*) - F(\xB^{t+1})]\leq (1-\epsilon)^T\EBB[F(\xB^*) - F(\xB^t)] - 2\bar{L}\epsilon D^2$. 
	In conclusion, we have
	$\EBB [F(\xB^T)]\geq (1-1/e)F(\xB^*) - 2\bar{L}\epsilon D^2$.

\section{Proof of Theorem \ref{thm_main_nm}}
	We note that \NMSCGPP shares the same structure as NMSCG except the gradient estimation.
	Following the same proof in Appendix H. of \cite{mokhtari2018stochastic}, we arrive at the following inequality ((113) in \cite{mokhtari2018stochastic})
	\begin{equation}
		F(\xB^{t+1}) - F(\xB^{t}) \geq \frac{1}{T}[(1-\frac{1}{T})^tF(\xB^*) - F(\xB^t)] - \frac{2D}{T}\|\nabla F(\xB^t) - \gB^t\| - \frac{\bar LD^2}{2 T^2}.
		\label{eqn_proof_nmscgpp_1}
	\end{equation}
	Recall the variance bound in Lemma \ref{lemma_variance_general}.
	By taking $\epsilon = 1/T$, we have $\EBB\|\nabla F(\xB^t) - \gB^t\|\leq \sqrt{\EBB\|\nabla F(\xB^t) - \gB^t\|^2}\leq \sqrt{2}\bar LD/T$.
	Take expectations on both sides of \eqref{eqn_proof_nmscgpp_1} and plug in the above variance bound to arrive
	\vspace{-2mm}
	\begin{equation}
		\EBB[F(\xB^{t+1})] \geq (1-\frac{1}{T})\EBB[F(\xB^{t})] + \frac{1}{T}(1-\frac{1}{T})^tF(\xB^*) - \frac{(4\sqrt{2}+1)\bar LD^2}{2T^2}.
		\label{eqn_proof_nmscgpp_2}
	\end{equation}
	Multiplying $(1-\frac{1}{T})^{-(t+1)}$ on both sides of \eqref{eqn_proof_nmscgpp_2}, we have
	\begin{equation*}
		(1-\frac{1}{T})^{-(t+1)}\EBB[F(\xB^{t+1})] \geq (1-\frac{1}{T})^{-t}\EBB[F(\xB^{t})] + \frac{1}{T}(1-\frac{1}{T})^{-1}F(\xB^*) - \frac{(4\sqrt{2}+1)\bar LD^2}{2T^2(1-\frac{1}{T})^{(t+1)}}.
		\label{eqn_proof_nmscgpp_3}
	\end{equation*}
	Sum the above inequality from $t=0$ to $T-1$ and use $F(\xB^{0}) \geq 0$ to obtain
	\begin{equation}\label{eqn_proof_nmscgpp_4}
	\begin{aligned}
	&(1-T^{-1})^{-T}\EBB[F(\xB^T)] \\
	\geq& F(\xB^{0}) + (1-T^{-1})^{-1}F(\xB^*) - {(2\sqrt{2}+0.5)\bar LD^2T^{-2}}\cdot[(1-T^{-1})^{-T}-1](T\!-\!1) \\
	\geq& (1-T^{-1})^{-1}F(\xB^*) - {(2\sqrt{2}+0.5)\bar LD^2T^{-2}}\cdot[(1-T^{-1})^{-T}-1](T-1)
	\end{aligned}
	\end{equation}
	Multiply $(1-T^{-1})^{T}$ on both sides of \eqref{eqn_proof_nmscgpp_4}
	\begin{align}
		\EBB[F(\xB^T)] &\geq (1-T^{-1})^{T-1}F(\xB^*) - {(2\sqrt{2}+0.5)\bar LD^2T^{-2}}\cdot[1-(1-T^{-1})^{T}](T-1) \notag \\
		&\geq e^{-1}F(\xB^*) - {(2\sqrt{2}+0.5)\bar LD^2T^{-1}}.
	\end{align}

\subsection{Multilinear Extension as Non-oblivious Stochastic Optimization} \label{section_multilinear_to_non_oblivious}
We proceed to show that the problem in \eqref{multilinear} is captured by \eqref{eqn_DR_submodular_maximization_loss}. To do so, 
use $\text{Ber}(b; m)$ with $b\in\{0, 1\}$ and $m\in[0, 1]$ to denote the Bernoulli distribution with parameter $m$, i.e., $\text{Ber}(b; m) = m^b(1-m)^{1-b}$.
Define $p(\zB,\gamma; \xB)$ as
\vspace{-3mm}
\begin{equation}\label{eqn_multilinear_extension_p}
p(\zB,\gamma; \xB) = p(\gamma)\times \prod_{i=1}^{d} \text{Ber}(\zB_i; \xB_i),
\vspace{-2mm}
\end{equation}
where $p(\gamma)$ is defined in (\ref{eqn_stochastic_submodular_function_definition}), $\zB_i$ is the $i^{th}$ entry of $\zB$, and $\xB_i$ is the $i^{th}$ entry of $\xB$.
Let $N(\zB)$ be a subset of $N$ such that $i\in N(\zB)$ iff $\zB_i = 1$.
We then define $\tilde{F}(\xB;\zB,\gamma)$ as
\vspace{-1mm}
\begin{equation}\label{eqn_multilinear_extension_f}
\tilde{F}(\xB; \zB,\gamma) = f_\gamma(N(\zB)),
\vspace{-1mm}
\end{equation}
where $f_\gamma$ is defined in (\ref{eqn_stochastic_submodular_function_definition}). 
For a fixed $\zB$, the stochastic function $\tilde{F}$ does not depend on $\xB$ and hence $\nabla \tilde{F}(\xB; \zB) = 0$.
By considering the definition of $\tilde{F}(\xB; \zB,\gamma)$ in \eqref{eqn_multilinear_extension_f}, the multilinear extension function $F$ in \eqref{multilinear}, and the probability distribution $p(\zB,\gamma; \xB)$ in \eqref{eqn_multilinear_extension_p} it can be verified that $F$ is the expectation of the random $\tilde{F}(\xB; \zB,\gamma)$, and, therefore, the problem in \eqref{multilinear} can be written as \eqref{eqn_DR_submodular_maximization_loss}. 

At the first glance, it seems that we can apply \SCGPP to maximize the multilinear extension function $F$. However, the smoothness conditions required for the result in Theorem~\ref{thm_main} do not hold in the multilinear setting. Following the result in Lemma~\ref{lemma_hessian_estimator}, we can derive an unbiased estimator for the second-order differential of \eqref{multilinear} using
\vspace{-2mm}
\begin{equation*}
\begin{aligned}
\tilde{\nabla}^2 F(\yB; \zB) =\tilde{F}(\yB;\zB)\left[[\nabla\log p(\zB,\gamma;\yB)] [\nabla\log p(\zB,\gamma;\yB)]^\top + \nabla^2\log p(\zB,\gamma;\yB)\right], \\
= f_\gamma(N(\zB))\left[[\sum_{i=1}^{d}\nabla\log \text{Ber}(\zB_i; \xB_i)] [\sum_{i=1}^{d}\nabla\log \text{Ber}(\zB_i; \xB_i)]^\top + \sum_{i=1}^{d}\nabla^2\log \text{Ber}(\zB_i; \xB_i)\right],
\end{aligned}\vspace{-2mm}
\end{equation*}
{{where we use $\nabla \tilde{F}(\xB; \zB) = 0$ in the first equality}} and use \eqref{eqn_multilinear_extension_p} and \eqref{eqn_multilinear_extension_f} in the second one.
Further, note that $[\nabla\log \text{Ber}(\zB_i; \xB_i)]^2+\nabla^2\log \text{Ber}(\zB_i; \xB_i) = 0$ for all $i\in[d]$ and hence, the above estimator can be further simplified to
\vspace{-2mm}
\begin{equation}\label{eqn_multilinear_extension_Hessian_estiamtor_1}
\tilde{\nabla}^2 F(\yB; \zB,\gamma) = f_\gamma(N(\zB))\sum_{i,j = 1}^{d}\mathbbm{1}_{i\neq j}[\nabla\log \text{Ber}(\zB_i; \xB_i)] [\nabla\log \text{Ber}(\zB_j; \xB_j)]^\top.\vspace{-2mm}
\end{equation}
Despite the simple form of \eqref{eqn_multilinear_extension_Hessian_estiamtor_1}, the smoothness property in Assumption \ref{ass_smooth} is absent since every entry in the matrix $\tilde{\nabla}^2 F(\yB; \zB,\gamma)$ can have unbounded second-order moment when $\xB_i \rightarrow 0$ or $\xB_i \rightarrow 1$.

\subsection{Detailed Implementation of \SCGPP for Multilinear Extension}

\begin{algorithm}[t]
	\caption{(\SCGPP) for Multilinear Extension}
	\label{algo_scg++_multi}
	\begin{algorithmic}[1]
		\REQUIRE Minibatch size $|\MM_0|$ and $|\MM|$, and total number of rounds $T$
		\FOR{$t = 1$ \TO $T$}
		\IF{$t = 1$}
		\STATE Sample $\MM_0$ of $(\gamma, \zB)$ according to $p(\zB, \gamma;\xB^0)$ and compute $\gB^0$ using \eqref{eqn_multilinear_extension_gradient_estimator};
		\ELSE
		\STATE Compute the Hessian approximation $\tilde{\nabla}_{\MM}^2 = \frac{1}{|\MM|} \sum_{k=1}^{|\MM|}\tilde{\nabla}_{k}^2$ based on \eqref{eqn_multilinear_extension_hessian_stochastic_ineqj};
		\STATE Construct $\tilde{\Delta}^t$ based on \eqref{eqn_gradient_difference_estimator};
		\STATE Update the stochastic gradient approximation $\gB^t := \gB^{t-1} + \tilde{\Delta}^t;$
		\ENDIF
		\STATE Compute the ascent direction $\vB^t := \argmax_{\vB\in\CM}\{\vB^\top\gB^t\}$;
		\STATE Update the variable $\xB^{t+1} := \xB^t + 1/{T}\cdot{\vB^t}$;
		\ENDFOR
	\end{algorithmic}
\end{algorithm}
We briefly mentioned the Hessian estimator $\tilde{\nabla}^2_k$ in \eqref{eqn_multilinear_extension_hessian_stochastic_ineqj}.
In this section, we describe \SCGPP for minimizing the Multilinear Extension \eqref{eqn_multilinear_extension_problem} in Algorithm \ref{algo_scg++_multi}.
In particular, we specify the gradient construction for $\xB^0$ by using the following equality
\begin{eqnarray}
	[\nabla F(\xB)]_i = F(\xB; \xB_i\leftarrow 1) - F(\xB; \xB_i\leftarrow 0).
	\label{eqn_multilinear_extension_gradient}
\end{eqnarray}
Since both terms in \eqref{eqn_multilinear_extension_gradient} are in expectation, we can sample a mini-batch $\MM_0$ of $(\gamma, \zB)$ from \eqref{eqn_multilinear_extension_gradient} to obtain an unbiased estimator of $\nabla F(\xB)$
\begin{equation}
	[\gB^0]_i \defi \frac{1}{|\MM_0|}\sum_{k=1}^{|\MM_0|} f_{\gamma_k}(N(\zB_k)\cup \{i\}) - f_{\gamma_k}(N(\zB_k)\setminus \{i\}).
	\label{eqn_multilinear_extension_gradient_estimator}
\end{equation}

\section{Proof of Lemma \ref{lemma_multilinear_extension_lemma}}
First note that we can write  the gradient $
	\nabla_{\xB_i}\log Ber(\zB_i; \xB_i) = \frac{\zB_i}{\xB_i} - \frac{1-\zB_i}{1-\xB_i}$.
We use $\zB_{\backslash i,j}$ to denote the random vector $\zB$ excluding the $i^{th}$ and $j^{th}$ entries, and denote by $\zB;\zB_i\leftarrow c_i, \zB_j\leftarrow c_j$ the vectors obtained by setting the $i^{th}$ and $j^{th}$ entries of $\zB$ to the corresponding $c_i$ and $c_j$.
Compute $\EBB_{\zB\sim p(\zB;\xB)}[\tilde{\nabla}^2 F(\yB; \zB,\gamma)]_{i,j}$ using \eqref{eqn_multilinear_extension_Hessian_estiamtor_1} 
\vspace{-1mm}
	\begin{align}
		\EBB_{\zB\sim p(\zB,\gamma;\xB)}[\tilde{\nabla}^2 F(\yB; \zB,\gamma)]_{i,j} 
		=& \EBB_{\zB}\left[f(N(\zB))[\nabla_{\xB_i} \! \log Ber(\zB_i; \xB_i)] [\nabla_{\xB_j}\log Ber(\zB_j; \xB_j)]\right]\notag\\
		 =& \sum_{c_i, c_j \in \{0, 1\}^2} \EBB_{\zB_{\backslash i,j}} f(N(\zB;\zB_i\leftarrow c_i, \zB_j\leftarrow c_j))(-1)^{c_i}(-1)^{c_j},
		\label{eqn_discrete_submodualr_Hessian_element}
	\end{align}\vspace{-2mm}
where in the first equality we use $\EBB_{\gamma} f_\gamma = f$ and in the second one uses
\vspace{-2mm}
\begin{equation}
	\xB_i^{c_i}\cdot(1-\xB_i)^{1-c_i}\cdot[\frac{\cB_i}{\xB_i} - \frac{1-\cB_i}{1-\xB_i}] = - (-1)^{c_i}.\vspace{-2mm}
\end{equation}
While there are four possible configurations for $c_i$ and $c_j$ in \eqref{eqn_discrete_submodualr_Hessian_element}, we discuss in detail the configuration of $c_i=c_j = 1$. The other three can be obtained similarly.
\vspace{-2mm}
\begin{equation}
	\EBB_{\zB_{\backslash i,j}} f(N(\zB;\zB_i\leftarrow 1, \zB_j\leftarrow 1)) = F(\yB; \yB_i\leftarrow 1, \yB_j \leftarrow 1),\vspace{-1mm}
\end{equation}
which recovers the first term in \eqref{eqn_multilinear_extension_lemma}.

\section{Proof of Theorem \ref{thm:multi}}
Theorem \ref{thm:multi} is identical to Theorem 2 of the conference version of this paper \cite{NIPS2019_9466}, hence we refer the reader to the detailed proof therein. Here we only provide a sketch.
By exploiting the sparsity of $\vB^{t}$ and the upper bound on the $\|\cdot\|_{2, \infty}$ norm of $\tilde{\nabla}_{k}^2$, we can obtain the following variance bound on $\gB^{t}$, which has an explicit dependence on the problem dimension $d$.
\begin{lemma}\label{kashk}
	Recall the constructions of the gradient estimator \eqref{g_def_rec} and the Hessian estimator \eqref{eqn_multilinear_extension_hessian_stochastic_ineqj}.
	In the multilinear extension problem \eqref{eqn_multilinear_extension_problem}, under Assumption \ref{ass_bounded_marginal_return}, we have the following variance bound  
	\vspace{-3mm}
	\begin{equation}
	\EBB [\|\gB^{t} - \nabla F(\xB^{t})\|^2] \leq \frac{4r^2d \cdot\epsilon}{|\MM|}  D_f^2 + \frac{dD_f^2}{|\MM_0|}.\vspace{-2mm}
	\end{equation}
\end{lemma}
By choosing $|\MM| = \frac{2}{\epsilon}$ and $|\MM_0| = \frac{1}{2r^2\epsilon^2}$, we have
$\EBB[\|\nabla F(\xB^t) - \gB^t\|^2] \leq 4r^2d \cdot\epsilon^2  D_f^2$.
Further, from Taylor's expansion we can prove
\vspace{-1mm}
\begin{equation}
F(\xB^{t+1}) - F(\xB^{t}) - \langle \nabla F(\xB^{t}), \xB^{t+1} - \xB^{t}\rangle \geq - \frac{1}{2}\sqrt{rdD_f^2}\cdot\|\xB^{t+1} - \xB^{t}\|^2.
\label{eqn_proof_discrete}
\end{equation}
\vspace{-1mm}
Following the update rule of \SCGPP, the function value can be proved to decrease in each iteration: for $T =\frac{1}{\epsilon}$
\vspace{-1mm}
\begin{equation*}
\EBB [F(\xB^{t+1})]\geq \EBB [F(\xB^t)] + \epsilon\EBB[F(\xB^*) - F(\xB^t)] - 6\sqrt{r^3dD_f^2}\cdot\epsilon^2,
\end{equation*}
which can be translated to 
$\EBB [F(\xB^*) - F(\xB^{\frac{1}{\epsilon}})]\leq (1-\epsilon)^{\frac{1}{\epsilon}}[F(\xB^*) - F(\xB^{0})] - 6\sqrt{r^3d}\cdot D_f\cdot\epsilon$.
In conclusion, we have
$\EBB[ F(\xB^{\frac{1}{\epsilon}})]\geq (1-1/e)F(\xB^*) - 6\sqrt{r^3d}\cdot D_f\cdot\epsilon$.

\section{Proof of Theorem~\ref{thm:nm-multi}}
	From \eqref{eqn_proof_discrete} we have 
	\vspace{-1mm}
	\begin{equation}
	F(\xB^{t+1}) - F(\xB^{t}) - \langle \nabla F(\xB^{t}), \xB^{t+1} - \xB^{t}\rangle \geq - \frac{1}{2}\sqrt{rdD_f^2}\cdot\|\xB^{t+1} - \xB^{t}\|^2.
	\end{equation}
	By using the above result and following the similar proof in Appendix H. of \cite{mokhtari2018stochastic}, we arrive at the following inequality
	\vspace{-2mm}
	\begin{equation}
	F(\xB^{t+1}) - F(\xB^{t}) \geq \frac{1}{T}[(1-\frac{1}{T})^tF(\xB^*) - F(\xB^t)] - \frac{2r}{T}\|\nabla F(\xB^t) - \gB^t\| - \frac{\sqrt{r^3dD_f^2}}{2 T^2}.
	\label{eqn_proof_nmdscgpp_1}
	\end{equation}
	By using Lemma \ref{kashk} and by choosing $|\MM| = \frac{2}{\epsilon}$ and $|\MM_0| = \frac{1}{2r^2\epsilon^2}$, we have
	\begin{equation}
	\EBB[\|\nabla F(\xB^t) - \gB^t\| ]\leq \sqrt{\EBB[\|\nabla F(\xB^t) - \gB^t\|^2]} \leq 2r\sqrt{d} \epsilon  D_f.
	\end{equation}
	Take expectation on both sides of \eqref{eqn_proof_nmdscgpp_1} and plug in the above variance bound to arrive
	\begin{equation}
	\EBB[F(\xB^{t+1})] \geq (1-\frac{1}{T})\EBB[F(\xB^{t})] + \frac{1}{T}(1-\frac{1}{T})^tF(\xB^*) - \frac{5}{2T^2} {\sqrt{r^3dD_f^2}}.
	\label{eqn_proof_nmdscgpp_2}
	\end{equation}
	By multiplying $(1-\frac{1}{T})^{-(t+1)}$ on both sides of \eqref{eqn_proof_nmdscgpp_2}, we have
	\vspace{-3mm}
	\begin{equation*}
	(1-\frac{1}{T})^{-(t+1)}\EBB[F(\xB^{t+1})] \geq (1-\frac{1}{T})^{-t}\EBB[F(\xB^{t})] + \frac{1}{T}(1-\frac{1}{T})^{-1}F(\xB^*) - \frac{5\sqrt{r^3dD_f^2}}{2T^2(1-\frac{1}{T})^{(t+1)}}.\vspace{-2mm}
	\end{equation*}
	Sum the above inequality from $t=0$ to $T-1$ and use $F(\xB^{0}) \geq 0$ to obtain
	\begin{align}
	&(1-T^{-1})^{-T}\EBB[F(\xB^T)] \notag\\
	\geq &F(\xB^{0}) + (1-T^{-1})^{-1}F(\xB^*) - {2.5T^{-2}\sqrt{r^3dD_f^2}}\cdot[(1-T^{-1})^{-T}-1](T-1) \notag \\
	\geq &(1-T^{-1})^{-1}F(\xB^*) - {2.5T^{-2}\sqrt{r^3dD_f^2}}\cdot[(1-T^{-1})^{-T}-1](T-1).
	\label{eqn_proof_nmdscgpp_4}
	\end{align}
	Multiply $(1-T^{-1})^{T}$ on both sides of \eqref{eqn_proof_nmdscgpp_4} to derive
	\begin{align}
	\EBB[F(\xB^T)] &\geq (1-T^{-1})^{T-1}F(\xB^*) - {2.5T^{-2}\sqrt{r^3dD_f^2}}\cdot[1-(1-T^{-1})^{T}](T-1) \notag \\
	&\geq e^{-1}F(\xB^*) - {2.5T^{-1}\sqrt{r^3dD_f^2}},
	\end{align}
	where we use $(1-T^{-1})^{T-1}\geq e^{-1}$.

\section{Proof of Theorem~\ref{sub_lower}}
Theorem \ref{thm:multi} is identical to Theorem 3 of the conference version of this paper \cite{NIPS2019_9466} and we refer the reader to the detailed proof therein. Here we only provide a sketch.
Our goal is to construct a submodular function $f$, defined through the expectation $f(S) = \mathbb{E}_{\gamma \sim p(\gamma)}[f_\gamma(S)]$, such that obtaining a $(1-1/e - \epsilon)$-optimal solution for maximizing $f$ under the $k$-cardinality constraint requires at least $\min\{\exp(\alpha(\epsilon) k), O(1/\epsilon^2)\}$ i.i.d. samples $f_\gamma(\cdot)$.  For maximizing monotone submodular functions under the $k$-cardinality constraint, we know that going beyond the the approximation factor $(1-1/e)$ is computationally hard. In other words, one can construct a specific monotone submodular set function, call it $f_0$, such that finding a $(1-1/e+\delta)$-optimal solution requires at least $\exp\{\alpha(\delta) k\}$ function queries. The main idea of the proof is to slightly change the value of $f_0$, by adding Bernoulli random variables whose success probabilities are small--say of order $\epsilon$--, such that the following holds: In order to obtain a $(1-1/e-\epsilon)$-optimal solution for the new function $f$ under the cardinality constraint, one would need to either find $(1-1/e+\epsilon/4)$-optimal solution for $f_0$ (which requires exponentially many samples), or to accurately estimate the parameters of the added Bernoulli random variables--a task that is known information-theoretically to require at least $O(1/\epsilon^2)$ i.i.d. samples from the Bernoulli random variables.  The function $f$ is the desired stochastic function of the theorem.

{{
\bibliographystyle{siamplain}
\bibliography{scgpp}
}}

\end{document}